\documentclass[onefignum,onetabnum]{siamart220329}

\usepackage{algorithmic}
\usepackage{geometry}

\usepackage{microtype}
\usepackage{graphicx}
\usepackage{booktabs} % for professional tables
\usepackage{microtype}
\usepackage{graphicx}
\usepackage{multirow}
\usepackage{booktabs} % for professional tables
\usepackage{hyperref}
\usepackage{subcaption}
\usepackage{pgfplots}
\usepackage{pgfplotstable}
\pgfplotsset{compat=newest}
\usetikzlibrary{spy}
\usepackage{hyperref}

\usepackage{pifont}
\usepackage{amsmath}
\usepackage{amssymb}
\usepackage{mathtools}
\usepackage{tabularx}
\numberwithin{equation}{section}

\newtheorem{assumption}[theorem]{Assumption}
\ifpdf
  \DeclareGraphicsExtensions{.eps,.pdf,.png,.jpg}
\else
  \DeclareGraphicsExtensions{.eps}
\fi
\allowdisplaybreaks[1]

\headers{ Stochastic Momentum ADMM}{Kangkang Deng et al.}

\title{Stochastic Momentum ADMM for nonconvex and nonsmooth optimization with application to PnP algorithm }

\author{
Kangkang Deng\thanks{ Department of Mathematics,  National University of Defense Technology, Changsha, 410073,
China (\email{freedeng1208@gmail.com}, \email{zhangshuchang19@nudt.edu.cn},\email{wangboyu20@nudt.edu.cn},\email{jinjiachen@nudt.edu.cn}) } 
\and Shuchang Zhang\footnotemark[1] \and  Boyu Wang\footnotemark[1] \and Jiachen Jin\footnotemark[1] \and Juan Zhou\footnotemark[1] \thanks{School of Mathematics and Computational Science, Xiangtan University, Xiangtan, 411105, China. (\email{juanzhou425@gmail.com})} \and Hongxia Wang\footnotemark[1] \thanks{Corresponding author. (\email{wanghongxia@nudt.edu.cn})}.
}

\usepackage{amsopn}

\ifpdf
\hypersetup{
  pdftitle={Stochastic Momentum ADMM},
  pdfauthor={Kangkang Deng}
}
\fi

\allowdisplaybreaks[2]
\begin{document}

\maketitle

\begin{abstract}
This paper proposes SMADMM, a single-loop Stochastic Momentum Alternating Direction Method of Multipliers for solving a class of nonconvex and nonsmooth composite optimization problems. SMADMM achieves the optimal oracle complexity of $\mathcal{O}(\epsilon^{-3/2})$ in the online setting. Unlike previous stochastic ADMM algorithms that require large mini-batches or a double-loop structure, SMADMM uses only $\mathcal{O}(1)$ stochastic gradient evaluations per iteration and avoids costly restarts. To further improve practicality, we incorporate dynamic step sizes and penalty parameters, proving that SMADMM maintains its optimal complexity without the need for large initial batches. We also develop PnP-SMADMM by integrating plug-and-play priors, and establish its theoretical convergence under mild assumptions. Extensive experiments on classification, CT image reconstruction, and phase retrieval tasks demonstrate that our approach outperforms existing stochastic ADMM methods both in accuracy and efficiency, validating our theoretical results.
\end{abstract}

% REQUIRED

\section{Introduction}
In this paper, we study a class of nonconvex and nonsmooth constrained optimization problems of the form:
\begin{equation}\label{prob}
    \min_{x,y} \mathbb{E}_{\xi \in \mathcal{D}}[f(x,\xi)] + h(y), \quad \text{s.t.} \quad Ax + By = c,
\end{equation}
where $f(x,\xi): \mathbb{R}^n \to \bar{\mathbb{R}}$ is continuously differentiable but not necessarily convex, and \( h : \mathbb{R}^d \to \bar{\mathbb{R}} \) is a convex function; $A\in \mathbb{R}^{p\times n}$ and $B\in\mathbb{R}^{p\times d}$; $\mathcal{D}$ is a distribution over an arbitrary space $\Xi$. We denote \( F(x) :=\mathbb{E}_{\xi \in \mathcal{D}}[f(x,\xi)] \).     
This formulation arises in a variety of machine learning applications, including  statistical learning \cite{boyd2011distributed}, distributed learning \cite{yang2022survey,mancino2023decentralized,mirzaeifard2024decentralized}, computer vision, and 3D CT image reconstruction \cite{barber2024convergence,he2018optimizing}, among others. In this paper, we focus on an online setting. Specifically, we do not know the entire function $F$, but we are allowed to access $f$ through a stochastic first-order  oracle (SFO), which returns a stochastic gradient at a queried point.  That is, given any $x$, we may compute $\nabla f(x, \xi)$ for some $\xi$ drawn i.i.d. from $\mathcal{D}$.  This SFO mechanism is particularly relevant in many online learning and expected risk minimization problems, where the noise in the SFO stems from the uncertainty inherent in sampling from the underlying streaming data.  Our primary interest lies in analyzing the oracle complexity, defined as the total number of queries to the SFO  required to attain an \(\epsilon\)-KKT point pair, as shown in Definition~\ref{def:kkt}.

A widely-used method for solving problem \eqref{prob} is the ADMM  \cite{glowinski1975approximation,gabay1976dual,boyd2011distributed,han2022survey}. The popularity of ADMM stems from its flexibility in splitting the objective into a loss term \(f\) and a regularizer \(h\), making it particularly effective for handling complex structured problems commonly encountered in machine learning. In recent years, stochastic variants of ADMM \cite{huang2016stochastic,huang2019faster,zheng2016fast, zeng2024unified,zheng2016stochastic, zeng2024accelerated} have been extensively studied, addressing both convex and nonconvex settings. These works primarily focus on improving iteration complexity by employing stochastic variance-reduced gradient estimators such as SVRG \cite{johnson2013accelerating} and SARAH \cite{Nguyen2017}, etc. However, these methods are typically restricted to the finite-sum setting.

A notable exception is SARAH-ADMM~\cite{huang2019faster}, which achieves an optimal oracle complexity of \( \mathcal{O}(\epsilon^{-3/2}) \) in the online setting. However, it suffers from a double-loop structure, requiring expensive large-batch gradients at each outer iteration. This significantly limits its practical deployment in streaming environments or when large batches are unavailable.
Moreover, existing stochastic ADMM methods often rely on fixed penalty parameters, which can severely affect performance and convergence. There is a lack of understanding on how to design adaptive penalty schedules while maintaining optimal theoretical guarantees.

\begin{table*}
 \caption{Comparison of the oracle complexity results of Online ADMM algorithms. The oracle complexity means the total number of queries to the SFO given in Definition \ref{def:stochastic}.    We do not list the work in \cite{zeng2024unified,zeng2024accelerated} since they focus on the finite-sum setting and do not apply to the online setting.
 }%\revise{Add a column on whether needs compute local full gradient Periodically, Require accessibility to full gradient?}}
    \centering
    \begin{tabular}{|c|c|c|c|c|}
  \hline
        Algorithm  & Batchsize &  Penalty parameter &  Single loop  & Oracle complexity \\ \hline 
           \cite{huang2016stochastic}     &   $\mathcal{O}(1)$ & fixed & \checkmark & $\mathcal{O}(\epsilon^{-2})$  \\ \hline
 \cite{huang2019faster}     &  $\mathcal{O}(\epsilon^{-1})$ or $\mathcal{O}(\epsilon^{-1/2})$ &  fixed & 
 \ding{55} & $\mathcal{O}(\epsilon^{-\frac{3}{2}})$\\\hline
         Ours (Theorem \ref{theorem:constant})  &  $\mathcal{O}(\epsilon^{-1/2})$ then $\mathcal{O}(1)$ & fixed & \checkmark & $\mathcal{O}(\epsilon^{-\frac{3}{2}})$\\ \hline
        Ours  (Theorem \ref{theorem:diminish})  &   $\mathcal{O}(1)$ &  dynamic & \checkmark & $\tilde{\mathcal{O}}(\epsilon^{-\frac{3}{2}})$\\ \hline
    \end{tabular}
    % \caption{Comparisons of oracle complexity for decentralized methods on manifolds}
    \label{tab:my_label}
\end{table*}

\subsection{Contributions}
We summarize our main contributions as follows:

\begin{itemize}

\item \textbf{Single-loop stochastic ADMM with optimal oracle complexity.}  
We propose SMADMM, a novel single-loop stochastic ADMM algorithm that leverages momentum-based gradient estimators~\cite{cutkosky2019momentum,levy2021storm}. SMADMM achieves the optimal oracle complexity of $\mathcal{O}(\epsilon^{-3/2})$ for nonconvex composite problems, using only $\mathcal{O}(1)$ stochastic samples per iteration (except for the first iteration, which requires a mini-batch of size $\mathcal{O}(\epsilon^{-1/2})$).
Unlike SARAH-ADMM~\cite{huang2019faster}, which relies on a double-loop structure with large batch sizes, SMADMM is the first single-loop stochastic ADMM algorithm to match the optimal oracle complexity in the online setting.

\item \textbf{SMADMM with dynamic penalty parameter.} 
To eliminate the need for large batch sizes, we further analyze SMADMM under time-varying parameters, including dynamic step sizes, momentum, and penalty parameters. We show that the algorithm still retains the optimal complexity of $\mathcal{O}(\epsilon^{-3/2})$.
Notably, SMADMM is the first stochastic ADMM method that supports dynamic penalty scheduling, enhancing both convergence and robustness. A detailed comparison of oracle complexities is presented in Table~\ref{tab:my_label}, where SMADMM consistently outperforms existing online stochastic ADMM algorithms~\cite{huang2016stochastic,huang2019faster}.

    \item \textbf{PnP-integrated stochastic ADMM.}  Finally, we extend our method by integrating it with PnP priors, resulting in the PnP-SMADMM algorithm. Under mild assumptions, we prove that PnP-SMADMM achieves the optimal oracle complexity of \(\mathcal{O}(\epsilon^{-\frac{3}{2}})\), outperforming existing PnP with stochastic (PnP-SADMM) algorithms. Numerical experiments on classification, CT image reconstruction and phase retrieve tasks demonstrate the practical effectiveness of our approach and validate the theoretical findings.

\end{itemize}

\subsection{Related works}

\textbf{Stochastic ADMM algorithm.}  Large-scale optimization problems  \eqref{prob} 
 typically involve a large sum of $N$ component functions, making it infeasible for deterministic ADMMs to compute the full gradient 
at each iteration. Early stochastic ADMM algorithms focus on  the convex case, such as \cite{ouyang2013stochastic,wang2013online,suzuki2013dual}. There are also many works for considering  variance reduction (VR) techniques into ADMM, including \cite{zhong2014fast,suzuki2014stochastic,zheng2016fast,xu2017admm,fang2017faster,liu2020accelerated}.  So far, the above discussed ADMM methods build on the convexity of objective functions. In fact, ADMM is also highly successful in solving various nonconvex problems such as tensor decomposition and training neural networks.   the nonconvex stochastic ADMMs \cite{huang2016stochastic,zheng2016stochastic} have been proposed with the VR techniques such as the SVRG \cite{johnson2013accelerating} and the SAGA \cite{defazio2014saga}. In
addition, \cite{huang2018mini} have extended the online/stochastic ADMM \cite{ouyang2013stochastic} to the nonconvex setting. \cite{huang2019faster} propose a SPIDER-ADMM by using a new stochastic path-integrated differential estimator
(SPIDER). \cite{zeng2024unified} propose a unified framework of inexact stochastic ADMM. \cite{zeng2024accelerated} propose an accelerated SVRG-ADMM algorithm (ASVRG-ADMM), which extends SVRG-ADMM by incorporating momentum techniques.  However, the method depends on a double-loop structure, necessitating large batch gradient calculations after each inner loop. This becomes impractical for real-time applications, particularly in scenarios like streaming or online learning, where the batch size cannot be controlled.

\textbf{PnP-type algorithms.}
Plug-and-play (PnP) \cite{venkatakrishnan2013plug,ahmad2020plug,kamilov2023plug} has emerged as a class of deep learning algorithms for solving inverse problems by
denoisers as image priors. PnP has been successfully used in many applications such as super-resolution,
phase retrieval, microscopy, and medical imaging \cite{zhang2019deep,metzler2018prdeep,zhang2017learning,wei2020tuning}. PnP draws an elegant connection between proximal methods and deep image models by replacing the proximity operator of $h$ with an image denoiser. These denoisers are used in various proximal algorithms such as HQS \cite{zhang2017learning,zhang2021plug}, ADMM and DRS \cite{romano2017little,ryu2019plug}, Proximal Gradient Descent (PGD) \cite{terris2020building}. To obtain the convergence of PnP algorithms, we need to add restrictions on deep denoiser, such as averaged \cite{sun2019online}, firmly nonexpansive \cite{sun2021scalable,terris2020building} or simply nonexpansive \cite{reehorst2018regularization,liu2021recovery}.   Another line of PnP work \cite{cohen2021has,hurault2021gradient,hurault2022proximal}  has explored the specification of the denoiser as a gradient-descent / proximal step  on a functional parameterized by a deep neural network.  Research on stochastic PnP algorithms remains relatively limited, with the stochastic PnP-ADMM  algorithms \cite{tang2020fast,sun2021scalable} being the most closely related work. However, these studies primarily focus on the case where $F$ is convex, which differs from the non-convex setting addressed in this work. 
\section{Preliminary}
%Augmented Lagrangian function, Stationary point, Assumption

Let us first define the approximated stationary point of \eqref{prob} based on the KKT condition. The Lagrangian function is defined as
\begin{equation}\nonumber
    \mathcal{L}(x, y, \lambda)=F(x)+h(y)-\langle\lambda, A x+B y-c\rangle. 
\end{equation}
We give the definition of  $\epsilon$-stationary point of \eqref{prob}.
\begin{definition}\label{def:kkt}
    Given  $\epsilon>0$, the point $(x^*,y^*,\lambda^*)$ is said
to be an $\epsilon$-stationary point of \eqref{prob}, if it holds that
\begin{equation}\nonumber
   \mathbb{E}\left[ \text{dist}^2(0,\partial L(x^*,y^*,\lambda^*))\right] \leq \epsilon,
\end{equation}
where $\text{dist}^2(0,\partial L) = \min_{z\in \partial L}\|z\|^2$, and  $\partial L(x,y,\lambda)$ is defined by
\begin{equation}\label{eq:parial}
    \partial L(x,y,\lambda):= \left[ 
    \begin{array}{c}
         \nabla_x L(x,y,\lambda)  \\
         \partial_y L(x,y,\lambda) \\
       Ax + By - c
    \end{array}
    \right].
\end{equation}
\end{definition}
Next, we review the standard ADMM for solving \eqref{prob}. The augmented Lagrangian function of \eqref{prob} is defined as
\begingroup
\fontsize{9pt}{10pt}\selectfont
$$
\mathcal{L}_\rho(x, y, \lambda)=F(x)+h(y)-\langle\lambda, A x+B y-c\rangle+\frac{\rho}{2}\|A x+B y-c\|^2
$$
\endgroup
where $\lambda$ is a Lagrange multiplier, and $\rho$ is a penalty parameter. At $t$-th iteration, the ADMM executes the following update:
$$
\left\{
\begin{aligned}
& y_{k+1}=\arg \min _y \mathcal{L}_\rho\left(x_k, y, \lambda_k\right) \\
& x_{k+1}=\arg \min _x \mathcal{L}_\rho\left(x, y_{k+1}, \lambda_k\right) \\
& \lambda_{k+1}=\lambda_k-\rho\left(A x_{k+1}+B y_{k+1}-c\right)
\end{aligned}
\right.
$$
When $F$ involves a large sum of $N$ component functions, the above ADMM algorithm requires the computation of the full gradient at each iteration, which becomes computationally infeasible. This limitation motivates the design of a stochastic ADMM algorithm for solving \eqref{prob}.

Finally, we present several assumptions for problem \eqref{prob}, which are consistent with those outlined in \cite{huang2019faster}.
\begin{assumption}\label{assm:lipsciz}
    Given any $\xi\in \mathcal{D}$,  the function $x \mapsto f(x,\xi)$ is $L$-smooth such that
$$
\mathbb{E}[ \left\|\nabla f(x,\xi)-\nabla f(y,\xi)\right\|] \leq L\|x-y\|, \forall x, y \in \mathbb{R}^n.
$$
\end{assumption} 

\begin{assumption}\label{assum:bound-grad}
    The stochastic gradient of loss function $f(x,\xi)$ is bounded, i.e., there exists a constant $\delta>0$ such that for all $x$ and $\xi\in \mathcal{D}$, it follows $\|\nabla f(x,\xi)\|^2 \leq \delta^2$.
\end{assumption} 
\begin{assumption}\label{assum:low-bound}
    $f(x)$ and $h(y)$ are all lower bounded, and let $f^*=\inf _x f(x)>-\infty$ and $h^*=$ $\inf _{y} h\left(y\right)>-\infty$.
\end{assumption}
\begin{assumption}\label{assum:rank-A}
    $A$ is a full row or column rank matrix. 
\end{assumption} 

\begin{assumption}\label{assum:variance}
    We assume access to a stream of independent random variables $\xi_1,\cdots,\xi_K\in \mathcal{D}$ such that for all $k$ and for all $x$, $\mathbb{E}[\nabla f(x,\xi)] = \nabla f(x)$.  We also assume there
is some $\sigma^2$ that upper bounds the noise on gradients: $\mathbb{E}[\|\nabla f(x,\xi) - \nabla f(x)\|^2]\leq \sigma^2$.
\end{assumption}
To measure the oracle complexity,  we give the definition of a stochastic first-order oracle (SFO) for \eqref{prob}.
\begin{definition}[\textbf{stochastic first-order oracle}]\label{def:stochastic}
    For the problem \eqref{prob}, a stochastic first-order oracle is defined as follows: compute the stochastic gradient $\nabla f(x,\xi)$ given a sample $\xi\in \mathcal{D}$.  
\end{definition}

\section{Stochastic Momentum ADMM}
This section gives our main algorithm, SMADMM, and presents the iteration complexity result. Since the $x$-subproblem and $y$-subproblem in standard ADMM are difficult to solve due to the existence of expected risk and matrix $B$, we maintain the update of $\lambda$ and change the $x$-subproblem and $y$-subproblem. 
To update the variable $y_{k+1}$, we introduce a proximal term $\frac{1}{2}\|y-y_k\|_{H}^2$ and solve the following subproblem:
\begin{equation}\nonumber
    y_{k+1} = \arg\min \mathcal{L}_{\rho_k}(x_k,y,\lambda_k) + \frac{1}{2}\|y-y_k\|_{H}^2,
\end{equation}
where $H \succ 0$ is a positive define matrix, and $\|y-y_k\|_{H}^2 = (y-y_k)^\top H (y-y_k)$.

For the $x$-subproblem, we first define an approximated function of the form:
\begin{equation}\label{func:alf}
\begin{aligned}
& \hat{\mathcal{L}}_{\rho}(x, y, \lambda, v, \bar{x})
= f(\bar{x})+v^T(x-\bar{x})+\frac{\eta_k}{2}\|x-\bar{x}\|_Q^2 \\
& -\langle\lambda, A x+B y-c\rangle+\frac{\rho}{2}\|A x+B y-c\|^2,
\end{aligned}
\end{equation}
where $v$ is a stochastic gradient estimator of $\nabla f$ at $x_k$ and $Q \succ 0$. Then we update $x_{k+1}$ by
\begingroup
\fontsize{8pt}{10pt}\selectfont
\begin{equation}\label{update-x1}
\begin{aligned}
& x_{k+1} = \arg\min_{x} \hat{\mathcal{L}}_{\rho_k}(x, y_{k+1}, \lambda_k, v_k, x_k) \\
 =& \left(\eta_k Q+{\rho_k} A^T A\right)^{-1}\left(\eta_k Q x_k-v_k- 
 {\rho_k} A^T\left(B y_{k+1}-c-\frac{\lambda}{{\rho_k}}\right)\right).
\end{aligned}
\end{equation}
\endgroup
When $A^T A$ is large, computing inversion of $\eta_k Q+{\rho_k} A^T A$ is expensive. To avoid it, we choose $Q=\left(I-\frac{{\rho_k}}{\eta_k} A^T A\right)$ to linearize it and \eqref{update-x1} reduced to 
\begin{equation}\label{eq:x-simply-update}
x_{k+1} \leftarrow x_k-\frac{1}{\eta_k}\left(v_{k}+{\rho_k} A^T\left(A x_k+B y_{k+1}-c-\frac{\lambda}{{\rho_k}}\right)\right).
\end{equation}
In this case, $\eta_k$ can be viewed as the stepsize for solving $x$-subproblem. Finally, we provide the update rule of $v_k$.  We focus on the following stochastic gradient estimator using the momentum technique introduced in \cite{cutkosky2019momentum}:
\begin{equation}\label{update:qik}
\begin{aligned}
    v_k = \nabla f(x_k,\xi_k) + (1-a_k)(v_{k-1} - \nabla f(x_{k-1},\xi_k)),
\end{aligned}
\end{equation}
where $a_k\in(0,1]$ is the momentum parameter.  We note that \eqref{update:qik} can be rewritten as 
\begin{equation}\label{update:qik1}
\begin{aligned}
    v_k = &  a_k \nabla f(x_k,\xi_k)  +(1-a_k)v_{k-1} \\
    &+(1-a_k)\nabla f(x_k,\xi_k) -  \nabla f(x_{k-1},\xi_k)),
\end{aligned}
\end{equation}
which hybrids stochastic gradient $ \nabla f(x_{k-1},\xi_k)$ with the recursive gradient estimator in \cite{Nguyen2017} for $a_k \in (0,1]$.  The detailed algorithm is referred to as Algorithm \ref{alg:sam}.

\begin{algorithm}[tb]
\caption{SMADMM}
\label{alg:sam}
\textbf{Input}: Parameters $a_k, \eta_k, m, {\rho_k}, H, Q$; initial points $x_0$, $y_0,z_0$. \\
\begin{algorithmic}[1] %[1] enables line numbers
\STATE Sample $\{\xi_{0,t}\}_{t=0}^m$ and let $v_0 = \frac{1}{m} \sum_{t=1}^{m} \nabla f(x_0,\xi_{0,t})$.
\FOR{$k = 0,\cdots,K-1$}
\STATE $ y_{k+1} = \arg\min_{y}\mathcal{L}_{{\rho_k}}(x_k,y_k,\lambda_k) + \frac{1}{2}\|y - y_k\|_H^2.$
\STATE $    x_{k+1}=  \arg\min_x \hat{\mathcal{L}}_{\rho_k}(x, y_{k+1}, \lambda_k, v_k,x_k).$
\STATE $      \lambda_{k+1}=\lambda_k-{\rho_k}\left(A x_{k+1}+B y_{k+1}-c\right).$
\STATE Sample $\xi_{k+1}\in \mathcal{D}$ and let
% \\
\begingroup
\fontsize{9pt}{10pt}\selectfont
\begin{equation*}
v_{k+1} = \nabla f(x_{k+1},\xi_{k+1}) + (1-a_{k+1})(v_{k} - \nabla f(x_{k},\xi_{k+1})).
\end{equation*}
\endgroup
\ENDFOR
\end{algorithmic}
%\textbf{Output}: $(x_K,y_K,\lambda_K)$.
\end{algorithm}

\subsection{The convergence result with constant parameters}\label{sec:constant}
Now we provide the main convergence result of our SMADMM algorithm. Let us first consider the case of constant stepsize and constant momentum parameters, i.e., 
$$\eta_k \equiv \eta,~~a_k \equiv a, ~~\rho_k \equiv \rho.$$
In particular, we show that under certain assumptions, SMADMM can achieve a oracle complexity of $\mathcal{O}(\epsilon^{-\frac{3}{2}})$. 
\begin{theorem}\label{theorem:constant}
Suppose that Assumptions \ref{assm:lipsciz}-\ref{assum:variance} hold. Let the sequence $\left\{x_k, y_k, \lambda_k\right\}_{k=1}^K$ be generated by Algorithm \ref{alg:sam}. Assume that
$$\rho_k \equiv \rho = c_{\rho}K^{1/3}, a_k \equiv a = c_{a}^2/\rho^2, \eta_k  \equiv \eta =  \frac{\phi_{\min } \rho \sigma_A }{20\phi_{\max}^2},$$
and $m = \lceil \rho\rceil$, where $\phi_{\min}$ and $\phi_{\max}$ denote the smallest and largest eigenvalues of positive definite matrix $Q$, $\sigma_A$ denotes the smallest eigenvalues of matrix $A A^T$, $c_{a}, c_{\rho}$ is two constants defined by
\begin{equation}\nonumber
    \begin{aligned}
        c_{a} & = \max\left\{ (\frac{1+2L^2}{2} + \frac{20L^2}{\sigma_A}) \frac{2}{\tau}, 1 \right\} ,\\
        c_{\rho} & = \max\{\frac{20L^2 + 2\sigma_A L }{\sigma_A\tau},\frac{\tau \sigma^2_{\max}(H)}{4 \|A\|^2 \|B\|^2 \sigma_{\min}(H)},1 \},
    \end{aligned}
\end{equation}
where $\tau  = \frac{ \phi_{\min }^2 \sigma_A }{40 \phi_{\max }^2}+\frac{\sigma_A }{2}$, $\sigma_{\min}(H)$ and $\sigma_{\max}(H)$ denote the smallest and largest eigenvalues of positive definite matrix $H$. Then we have that
\begin{equation}\nonumber
\begin{aligned}
   & \min_{1\leq k \leq K} \mathbb{E}\left[ \text{dist}^2(0,\partial L(x_k,y_k,\lambda_k))\right] \\
    \leq &  \mathcal{H}_1 K^{-2/3} +  \mathcal{H}_2 K^{-4/3} + \mathcal{H}_3 K^{-2},
    \end{aligned}
\end{equation}
where $\mathcal{H}_1,\mathcal{H}_2$ and $\mathcal{H}_3$ are constants defined in the Appendix \ref{appen:constant}.  As a consequence, Algorithm \ref{alg:sam} obtains an $\epsilon$-stationary point with at most
    $$
    \mathcal{K}: = \mathcal{O}(\max\{ \mathcal{K}_1, \mathcal{K}_2, \mathcal{K}_3 \})
    $$
    iterations. Here, $\mathcal{K}_1,\mathcal{K}_2,\mathcal{K}_3$ are given as follows:
    \begin{equation}\nonumber
        \begin{aligned}
            \mathcal{K}_1: & =  \mathcal{H}_1^{1.5} \epsilon^{-\frac{3}{2}}, 
            \mathcal{K}_2: = \mathcal{H}_2^{3/4} \epsilon^{-3/4},\mathcal{K}_3: = \mathcal{H}_3^{1/2} \epsilon^{-\frac{1}{2}}.
        \end{aligned}
    \end{equation}
\end{theorem}

According to Theorem \ref{theorem:constant}, our algorithm achieves an oracle complexity of \(\mathcal{O}(\epsilon^{-\frac{3}{2}})\), which outperforms existing methods such as \cite{huang2016stochastic}, where the oracle complexity is at best \(\mathcal{O}(\epsilon^{-2})\). Furthermore, compared to the approach in \cite{huang2019faster}, our method only requires an initial sample size of \(m = \mathcal{O}(\epsilon^{-1/2})\).

\subsection{The convergence result with dynamic parameters}\label{sec:diminish}
To mitigate the impact of the initial sample size, we extend our analysis to the case where both the stepsize and the momentum parameter are updated dynamically. The following theorem establishes that, even with an initial sample size of \(\mathcal{O}(1)\), our algorithm achieves the same oracle complexity as stated in Theorem \ref{theorem:constant}.

\begin{theorem}\label{theorem:diminish}
Suppose that Assumptions \ref{assm:lipsciz}-\ref{assum:variance} hold. Let the sequence $\left\{x_k, y_k, \lambda_k\right\}_{k=1}^K$ be generated by Algorithm \ref{alg:sam}. Assume that
$$\rho_k = c_{\rho}k^{1/3}, a_{k+1} = c_ak^{-2/3}, \eta_k = c_{\eta}k^{1/3},$$
and $m = 1$, 
%Let $\nu_k = c_{\nu}/\rho, \rho = c_{\rho}K^{1/3}$ and $a_{k+1} = c_ak^{-2/3}, \gamma_{k+1} = c_{\gamma}k^{1/3}, \eta_k = c_{\eta}k^{1/3}, m = 1$, 
where $ c_{\rho}, c_a, c_\eta$ are constants satisfying:
\begin{equation}\nonumber
    \begin{aligned}
       c_{\rho} & \geq \frac{8L}{\sigma_A} + \frac{160L^2}{\sigma_A^2} + \frac{\|A\|\|B\|}{\sigma^2_{\max}(H)},\\
        c_a & \geq \frac{3c_{\nu}c_{\rho} + 60 + 2c_{\gamma}\sigma_A c_{\rho}}{3c_{\gamma}\sigma_A c_{\rho}}, c_{\eta} \leq \frac{\sigma_A c_{\rho}}{\sqrt{160}\phi_{\max}}.  
    \end{aligned}
\end{equation}
Then we have that
\begin{equation}\nonumber
\begin{aligned}
  &  \min_{1\leq k \leq K} \mathbb{E}\left[ \text{dist}^2(0,\partial L(x_k,y_k,\lambda_k))\right] \\
    \leq & (\mathcal{G}_1+\mathcal{G}_3) K^{-2/3} +  \mathcal{G}_2 K^{-1},
    \end{aligned}
\end{equation}
where $\mathcal{G}_1,\mathcal{G}_2$ and $\mathcal{G}_3$ are constants dependent on a logaithmic factor of $K$, which are defined in the Appendix \ref{appendix-diminish}.  As a consequence, Algorithm \ref{alg:sam} obtains an $\epsilon$-stationary point with at most
    $$
    \mathcal{K}: = \mathcal{O}(\max\{ \mathcal{K}_4, \mathcal{K}_5 \})
    $$
    iterations. Here, $\mathcal{K}_4,\mathcal{K}_5$ are given as follows:
    \begin{equation}\nonumber
        \begin{aligned}
            \mathcal{K}_4: & =  (\mathcal{G}_1+\mathcal{G}_3)^{1.5} \epsilon^{-\frac{3}{2}}, 
            \mathcal{K}_5: = \mathcal{G}_2 \epsilon^{-1}.
        \end{aligned}
    \end{equation}
\end{theorem}

As established in Theorem \ref{theorem:diminish}, when dynamic parameters are considered, our algorithm attains an oracle complexity of \( \tilde{\mathcal{O}}(\epsilon^{-\frac{3}{2}}) \), which matches the result in Theorem \ref{theorem:constant} up to an additional logarithmic factor. Notably, the result in Theorem \ref{theorem:diminish} eliminates the need for a condition on the sampling number in the initial iteration, i.e., \( m = \mathcal{O}(\epsilon^{-1/2}) \), requiring only \( m = \mathcal{O}(1) \).

\section{The proof of main results}
\subsection{Common lemmas}
This section gives some common lemmas, which is useful for the subsequent  analysis. Here, we will not add any restriction for parameters $a_k,\eta_k$ and $\rho_k$. 
\begin{lemma}[\cite{xu2015augmented}, Lemma 2]\label{lem:curr}

Let $u_k$ and $w_k$ be two positive scalar sequences such that for all $k\geq 1$
\begin{equation}\label{2}
  u_{k} \leq \eta u_{k-1} + w_{k-1},
\end{equation}
where $\eta\in (0,1)$ is the decaying factor.  Then we have 
\begin{equation}\label{22}
       \sum_{k=0}^K u_k \leq \frac{u_0}{1-\eta} +  \frac{1}{1-\eta}  \sum_{k=0}^{K-1} w_{k}. 
\end{equation}
\end{lemma}

\begin{lemma}\label{lem:diff-nablaf-v}
   Suppose that Assumptions \ref{assm:lipsciz}-\ref{assum:variance} hold, and define 
$\varepsilon_k: = \nabla f(x_k) - v_k$.  Algorithm \ref{alg:sam} generates stochastic gradient $\left\{v_k\right\}$ satisfies
$$
\mathbb{E}[\left\|\varepsilon_k\right\|^2] \leq \left(1-a_k\right)^2 \mathbb{E}[\left\|\varepsilon_{k-1}\right\|^2]+2a_k^2 \sigma^2+2 L^2\left(1-a_k\right)^2 \mathbb{E}[\left\|x_k-x_{k-1}\right\|^2].
$$

\end{lemma}

\begin{proof}
   Let us denote $\mathcal{F}_k = \{\xi_0,\xi_1,\cdots,\xi_{k-1} \}$. From the definition of $\varepsilon_k$, we can write
$$
\begin{aligned}
\mathbb{E}\left[\|\varepsilon_k\|^2 \vert \mathcal{F}_k \right]= & \mathbb{E}\left[\|\nabla f\left(x_k, \xi_k\right)+\left(1-a_k\right)\left(v_{k-1}-\nabla f\left(x_{k-1}, \xi_k\right)\right)-\nabla f\left(x_k\right)\|^2 \vert \mathcal{F}_k \right] \\
= & \mathbb{E}\left[  \|a_k\left(\nabla f\left(x_k, \xi_k\right)-\nabla f\left(x_k\right)\right)+\left(1-a_k\right)\left(v_{k-1}-\nabla f\left(x_{k-1}\right)\right) \right. \\
&  \left.+\left(1-a_k\right)\left(\nabla f\left(x_k, \xi_k\right)-\nabla f\left(x_{k-1}, \xi_k\right)-\nabla f\left(x_k\right)+\nabla f\left(x_{k-1}\right)\right) \vert \mathcal{F}_k \|^2 \right] \\
\leq & \left(1-a_k\right)^2   \left\|\varepsilon_{k-1}\right\|^2 +2a_k^2 \mathbb{E}[ \left\|\nabla f\left(x_k, \xi_k\right)-\nabla f\left(x_k\right)\right\|^2 \vert \mathcal{F}_k ] \\
& +2\left(1-a_k\right)^2 \mathbb{E}[ \left\|\nabla f\left(x_k, \xi_k\right)-\nabla f\left(x_{k-1}, \xi_k\right)-\nabla f\left(x_k\right)+\nabla f\left(x_{k-1}\right)\right\|^2 \vert \mathcal{F}_k] \\
\leq & \left(1-a_k\right)^2 \left\|\varepsilon_{k-1}\right\|^2 +2a_k^2 \sigma^2+2\left(1-a_k\right)^2 \mathbb{E}[ \left\|\nabla f\left(x_k, \xi_k\right)-\nabla f\left(x_{k-1}, \xi_k\right)\right\|^2 \vert \mathcal{F}_k] \\
\leq & \left(1-a_k\right)^2 
 \left\|\varepsilon_{k-1}\right\|^2 +2a_k^2 \sigma^2+2 L^2\left(1-a_k\right)^2 \left\|x_k-x_{k-1}\right\|^2.
\end{aligned}
$$
where the first inequality uses unbiasedness of stochastic gradient $\nabla f(x_k,\xi_k)$ and $\|a+b\|^2 \leq 2\|a\|^2 + 2\|b\|^2$, the second inequality follows from Assumptions \ref{assum:variance} and  $\mathbb{E}\|x-\mathbb{E}(x)\|^2 \leq \mathbb{E}\|x\|^2$, the last inequality follows from  Assumption \ref{assm:lipsciz}. The conclusion of this lemma follows from taking expectation on both sides of this inequality.
\end{proof}

 First, given the sequence $\left\{x_k, y_k, \lambda_k\right\}_{k=1}^K$ generated by Algorithm \ref{alg:sam}, we give the upper bound of $\mathbb{E}\left\|\lambda_{k+1}-\lambda_k\right\|^2$.
\begin{lemma}
Let Assumptions \ref{assm:lipsciz}-\ref{assum:variance} hold. Suppose the sequence $\left\{x_k, y_k, \lambda_k\right\}_{k=1}^K$ is generated by the Algorithm \ref{alg:sam}. The following inequality holds
\begin{equation}\label{eq:diff-lambda-bound2}
\begin{aligned}
\mathbb{E}\left\|\lambda_{k+1}-\lambda_k\right\|^2 \leq & \frac{5}{\sigma_A} \mathbb{E}\left\|v_k-\nabla f\left(x_k\right)\right\|^2+\frac{5}{\sigma_A} \mathbb{E}\left\|v_{k-1}-\nabla f\left(x_{k-1}\right)\right\|^2+\frac{5 \eta_k^2 \phi_{\max }^2}{\sigma_A}\mathbb{E}[\left\|x_k-x_{k+1}\right\|^2] \\
& +\frac{5\left(L^2+\eta_{k-1}^2 \phi_{\max }^2\right)}{\sigma_A}\left\|x_{k-1}-x_k\right\|^2.
\end{aligned}
\end{equation}
where $\sigma_A$ denotes the smallest eigenvalues of matrix $A A^T$, and $\phi_{\max }$ denotes the largest eigenvalues of positive definite matrix $Q$.
\end{lemma}

\begin{proof}
     By the optimal condition of step 6 in Algorithm \ref{alg:sam}, we have
$$
\begin{aligned}
0 & =v_k-A^T \lambda_k+\rho A^T\left(A x_{k+1}+B y_{k+1}-c\right)-\eta_k Q\left(x_k-x_{k+1}\right) \\
& =v_k-A^T \lambda_{k+1}-\eta_k Q\left(x_k-x_{k+1}\right),
\end{aligned}
$$
where the second equality is due to step 7 in Algorithm \ref{alg:sam}. Thus, we have
\begin{equation}\label{eq:ALambda}
    A^T \lambda_{k+1}=v_k-\eta_k Q\left(x_k-x_{k+1}\right).
\end{equation}
By \eqref{eq:ALambda}, we have
\begin{equation}\label{eq:diff-lambda-bound}
\begin{aligned}
& \left\|\lambda_{k+1}-\lambda_k\right\|^2 \leq \sigma_A^{-1}\left\|A^T \lambda_{k+1}-A^T \lambda_k\right\|^2 \\
& \leq \sigma_A^{-1}\left\|v_k - v_{k-1}-\eta_k Q\left(x_k-x_{k+1}\right)+\eta_{k-1} Q\left(x_{k-1}-x_k\right)\right\|^2 \\
& =\sigma_A^{-1}\left\|v_k-\nabla f\left(x_k\right)+\nabla f\left(x_k\right)-\nabla f\left(x_{k-1}\right)+\nabla f\left(x_{k-1}\right)-v_{k-1}-\eta_k Q\left(x_k-x_{k+1}\right)+\eta_{k-1} Q\left(x_{k-1}-x_k\right)\right\|^2 \\
& \leq \frac{5}{\sigma_A}\left\|v_k-\nabla f\left(x_k\right)\right\|^2+\frac{5}{\sigma_A}\left\|v_{k-1}-\nabla f\left(x_{k-1}\right)\right\|^2+\frac{5 \eta_k^2 \phi_{\max }^2}{\sigma_A}\left\|x_k-x_{k+1}\right\|^2 \\
& \quad+\frac{5\left(L^2+\eta_{k-1}^2 \phi_{\max }^2\right)}{\sigma_A}\left\|x_{k-1}-x_k\right\|^2
\end{aligned}
\end{equation}
where the last inequality follows from the Assumption \ref{assm:lipsciz} and $\|Q(x-y)\|^2 \leq \phi_{\max }^2\|x-y\|^2$, where $\phi_{\max }$ denotes the largest eigenvalue of positive matrix $Q$. Taking expectation conditioned on information $\xi_k$ to \eqref{eq:diff-lambda-bound}, we complete the proof.

\end{proof}

\begin{lemma}
 Suppose that Assumptions \ref{assm:lipsciz}-\ref{assum:variance} hold. Let the sequence $\left\{x_k, y_k, \lambda_k\right\}_{k=1}^K$ be generated by Algorithm \ref{alg:sam}. Then
    \begin{equation}\label{eq:single-descent}
\begin{aligned}
\mathbb{E}\left[\mathcal{L}_{\rho_{k+1}}\left(x_{k+1}, y_{k+1}, \lambda_{k+1}\right)\right] \leq & \mathbb{E}\left[\mathcal{L}_{\rho_k}\left(x_k, y_k, \lambda_k\right)\right]+(\frac{1}{{\rho_k}}+\frac{\rho_{k+1} - \rho_k}{2 \rho_k^2}) \mathbb{E} [\left\|\lambda_{k+1}-\lambda_k\right\|^2] + \frac{\nu_k}{2}\mathbb{E}[\|v_k-\nabla f\left(x_k\right)\|^2] \\
&- \sigma_{\min}(H) \mathbb{E}[\|y_{k+1} - y_k\|^2]-\left(\eta_k \phi_{\min }+\frac{\sigma_A {\rho_k}}{2}-\frac{L}{2} - \frac{1}{2\nu_k}\right)\mathbb{E}[ \left\|x_{k+1}-x_k\right\|^2],
\end{aligned}
\end{equation}
where $\phi_{\min }$ denote the smallest eigenvalue of $Q$, $\nu_k>0$ is any positive real number and $\sigma_A$ denote the smallest eigenvalues of matrix $AA^\top$.

\end{lemma}

\begin{proof}
By the step 7 in Algorithm \ref{alg:sam}, we have 
\begin{equation}\label{eq:descent-y}
    \mathcal{L}_{{\rho_k}}(x_k,y_{k+1},\lambda_k) \leq \mathcal{L}_{{\rho_k}}(x_k,y_k,\lambda_k) - \sigma_{\min}(H) \|y_{k+1} - y_k\|^2.
\end{equation}
By the optimal condition of step 8 in Algorithm \ref{alg:sam}, we have
\begin{equation}\label{eq:lip-descent}
\begin{aligned}
& 0=\left(x_k-x_{k+1}\right)^T\left[v_k-A^T \lambda_k+{\rho_k}\left(A x_{k+1}+B y_{k+1}-c\right)-\eta_k Q\left(x_k-x_{k+1}\right)\right] \\
& =\left(x_k-x_{k+1}\right)^T\left[v_k-\nabla f\left(x_k\right)+\nabla f\left(x_k\right)-A^T \lambda_k+{\rho_k} A^T\left(A x_{k+1}+B y_{k+1}-c\right)-\eta_k Q\left(x_k-x_{k+1}\right)\right] \\
& \stackrel{(i)}{\leq} f\left(x_k\right)-f\left(x_{k+1}\right)+\left(x_k-x_{k+1}\right)^T\left(v_k-\nabla f\left(x_k\right)\right)+\frac{L}{2}\left\|x_{k+1}-x_k\right\|^2-\eta_k\left\|x_{k+1}-x_k\right\|_Q^2 \\
& -\lambda_k^T\left(A x_k-A x_{k+1}\right)+{\rho_k}\left(A x_k-A x_{k+1}\right)^T\left(A x_{k+1}+B y_{k+1}-c\right) \\
& \stackrel{(i i)}{=} f\left(x_k\right)-f\left(x_{k+1}\right)+\left(x_k-x_{k+1}\right)^T\left(v_k-\nabla f\left(x_k\right)\right)+\frac{L}{2}\left\|x_{k+1}-x_k\right\|^2-\eta_k\left\|x_{k+1}-x_k\right\|_Q^2 \\
& -\lambda_k^T\left(A x_k+B y_{k+1}-c\right)+\lambda_k^T\left(A x_{k+1}+B y_{k+1}-c\right)+\frac{{\rho_k}}{2}\left\|A x_k+B y_{k+1}-c\right\|^2 \\
& -\frac{{\rho_k}}{2}\left\|A x_{k+1}+B y_{k+1}-c\right\|^2-\frac{{\rho_k}}{2}\left\|A x_k-A x_{k+1}\right\|^2 \\
& =\mathcal{L}_{\rho_k}\left(x_k, y_{k+1}, \lambda_k\right)-\mathcal{L}_{\rho_k}\left(x_{k+1}, y_{k+1}, \lambda_k\right)+\left(x_k-x_{k+1}\right)^T\left(v_k-\nabla f\left(x_k\right)\right) \\
& +\frac{L}{2}\left\|x_{k+1}-x_k\right\|^2-\eta_k\left\|x_{k+1}-x_k\right\|_Q^2-\frac{{\rho_k}}{2}\left\|A x_k-A x_{k+1}\right\|^2 \\
& \stackrel{(iii)}{\leq} \mathcal{L}_{\rho_k}\left(x_k, y_{k+1}, \lambda_k\right)-\mathcal{L}_{\rho_k}\left(x_{k+1}, y_{k+1}, \lambda_k\right)+ \frac{\nu_k}{2}\|v_k-\nabla f\left(x_k\right)\|^2 \\
& -\left(\eta_k \phi_{\min }+\frac{\sigma_A {\rho_k}}{2}-\frac{L}{2} - \frac{1}{2\nu_k} \right)\left\|x_k-x_{k+1}\right\|^2,
\end{aligned}
\end{equation}
where the inequality ($i$) holds by the Assumption \ref{assm:lipsciz}; the equality ($ii$) holds by using the equality $(a-b)^Tb=\frac{1}{2}\left(\|a\|^2-\|a-b\|^2-\|b\|^2\right)$ on the term ${\rho_k}\left(A x_k-A x_{k+1}\right)^T\left(A x_{k+1}+B y_{k+1}-c\right)$; the inequality ($iii$) holds by using $-\phi_{\min }\left\|x_{k+1}-x_k\right\|^2 \geq-\left\|x_{k+1}-x_k\right\|_Q^2$ and $-\sigma_A\left\|x_{k+1}-x_k\right\|^2 \geq-\| A x_k-$ $A x_{k+1} \|^2$. Then taking expectation conditioned on information $\xi_k$ to \eqref{eq:lip-descent},  we have
\begin{equation}\label{eq:descent-x}
\mathbb{E}\left[\mathcal{L}_{\rho_k}\left(x_{k+1}, y_{k+1}, \lambda_k\right)\right] \leq \mathbb{E}[\mathcal{L}_{\rho_k}\left(x_k, y_{k+1}, \lambda_k\right)]+ \frac{\nu_k}{2}\mathbb{E}[\|v_k-\nabla f\left(x_k\right)\|^2]-\left(\eta_k \phi_{\min }+\frac{\sigma_A {\rho_k}}{2}-\frac{L}{2} - \frac{1}{2\nu_k}\right) \mathbb{E}[\left\|x_{k+1}-x_k\right\|^2]
\end{equation}
By the step 9 of Algorithm \ref{alg:sam}, and taking expectation conditioned on information $\xi_k$, we have
\begin{equation}\label{eq:descent-lambda}
\begin{aligned}
& \mathbb{E}\left[\mathcal{L}_{\rho_k}\left(x_{k+1}, y_{k+1}, \lambda_{k+1}\right)-\mathcal{L}_{\rho_k}\left(x_{k+1}, y_{k+1}, \lambda_k\right)\right]=\frac{1}{{\rho_k}} \mathbb{E}\left\|\lambda_{k+1}-\lambda_k\right\|^2.
%& \stackrel{(i)}{\leq} \frac{5 L^2}{\sigma_A \rho} \mathbb{E}\left\|x_t-\tilde{x}\right\|^2+\frac{5 L^2}{\sigma_A \rho}\left\|x_{t-1}-\tilde{x}\right\|^2+\frac{5 \eta_k^2 \phi_{\max }^2}{\sigma_A \rho} \mathbb{E}\left\|x_{t+1}-x_t\right\|^2 \\
%& +\frac{5\left(L^2+\eta_k^2 \phi_{\max }^2\right)}{\sigma_A \rho}\left\|x_t-x_{t-1}\right\|^2,
\end{aligned}
\end{equation}
In addition, replacing $\rho_k$ by $\rho_{k+1}$ in $\mathbb{E}\left[\mathcal{L}_{\rho_k}\left(x_{k+1}, y_{k+1}, \lambda_{k+1}\right)\right]$ yields 
\begin{equation}\label{eq:descent-rho}
\begin{aligned}
    \mathbb{E}\left[\mathcal{L}_{\rho_{k+1}}\left(x_{k+1}, y_{k+1}, \lambda_{k+1}\right)\right] & \leq  \mathbb{E}[\mathcal{L}_{\rho_k}\left(x_{k+1}, y_{k+1}, \lambda_{k+1}\right)] + \frac{\rho_{k+1} - \rho_k}{2} \| Ax_{k+1} +By_{k+1} - c \|^2\\
    & \leq  \mathbb{E}[\mathcal{L}_{\rho_k}\left(x_{k+1}, y_{k+1}, \lambda_{k+1}\right)] + \frac{\rho_{k+1} - \rho_k}{2 \rho_k^2} \| \lambda_{k+1} - \lambda_k \|^2.
    \end{aligned}
\end{equation}
Combining \eqref{eq:descent-y}, \eqref{eq:descent-x}, \eqref{eq:descent-lambda} with \eqref{eq:descent-rho} gives \eqref{eq:single-descent}. The proof is completed. 
\end{proof}

Finally, we give the upper bounds to the terms \eqref{eq:parial} in the optimality condition using $\|x_k - x_{k-1}\|^2$.
\begin{lemma}\label{lem:relation-optima-x}
    Suppose that Assumptions \ref{assm:lipsciz}-\ref{assum:variance} hold. Let the sequence $\left\{x_k, y_k, \lambda_k\right\}_{k=1}^K$ be generated by Algorithm \ref{alg:sam}. Then
    \begin{align}
        \left\|A^T \lambda_k-\nabla f\left(x_k\right)\right\|^2 & \leq 3  \|v_{k-1} - \nabla f(x_{k-1})\|^2   + 3(L^2 + \eta_{k-1}^2 \phi_{\max }^2)  \|x_k - x_{k-1}\|^2, \label{eq:bound-alambda} \\
        \operatorname{dist}^2\left(B^T \lambda_k, \partial h\left(y_k\right)\right) & \leq 2\rho_{k-1}^2\|B\|_2^2\|A\|_2^2\left\|x_k-x_{k-1}\right\|^2 + 2\sigma^2_{\max}(H) \|y_k - y_{k-1}\|^2,\label{eq:bound-partial}\\
        \left\|A x_{k}+B y_{k}-c\right\|^2 & =\frac{1}{\rho_{k-1}^2}\left\|\lambda_{k}-\lambda_{k-1}\right\|^2. \label{eq:bound-AB}
    \end{align}
\end{lemma}

\begin{proof}

It follows from \eqref{eq:ALambda} that
$$
\begin{aligned}
& \left\|A^T \lambda_k-\nabla f\left(x_k\right)\right\|^2 \\
& =\left\|v_{k-1}-\nabla f\left(x_k\right)-\eta_{k-1} Q\left(x_{k-1}-x_k\right)\right\|^2 \\
& =\left\|v_{k-1}-\nabla f\left(x_{k-1}\right)+\nabla f\left(x_{k-1}\right)-\nabla f\left(x_k\right)-\eta_{k-1} Q\left(x_{k-1}-x_k\right)\right\|^2 \\
& \leq 3( [ \|v_{k-1} - \nabla f(x_{k-1})\|^2 ]  + (L^2 + \eta_{k-1}^2 \phi_{\max }^2)  \|x_k - x_{k-1}\|^2  )
\end{aligned}
$$
By step 7 of Algorithm \ref{alg:sam}, there exists a sub-gradient $\mu \in \partial h\left(y_k\right)$ such that
$$
\begin{aligned}
& \operatorname{dist}^2\left(B^T \lambda_k, \partial h\left(y_k\right)\right)  \leq\left\|\mu-B^T \lambda_k\right\|^2 \\
 = & \left\|B^T \lambda_{k-1}-{\rho_{k-1}} B^T\left(A x_{k-1}+B y_k-c\right) - H (y_k - y_{k-1}) -B^T \lambda_k\right\|^2 \\
\leq  & 2\rho_{k-1}^2\|B\|_2^2\|A\|_2^2\left\|x_k-x_{k-1}\right\|^2 + 2\sigma^2_{\max}(H) \|y_k - y_{k-1}\|^2.
%& \leq {\rho_{k-1}}^2\|B\|_2^2\|A\|_2^2 \theta_k
\end{aligned}
$$
Finally, \eqref{eq:bound-AB} follows from the step 9 of Algorithm \ref{alg:sam}.
The proof is completed.

\end{proof}

\subsection{Proof of Section \ref{sec:constant}}\label{appen:constant}

We first show that the term $\sum_{k=1}^K\mathbb{E}[ \left\|x_{k+1}-x_k\right\|^2]$ can be bounded. 

\begin{lemma}\label{lem:xk-xk+1}
     Suppose that Assumptions \ref{assm:lipsciz}-\ref{assum:variance} hold. Let the sequence $\left\{x_k, y_k, \lambda_k\right\}_{k=1}^K$ be generated by Algorithm \ref{alg:sam} and 
     $$\rho_k \equiv \rho = c_{\rho}K^{1/3}, a_k \equiv a = c_{a}^2/\rho^2, \eta_k  \equiv \eta =  \frac{\phi_{\min } \rho \sigma_A }{20\phi_{\max}^2},~m = \lceil \rho\rceil,$$ where $\phi_{\min}$ and $\phi_{\max}$ denote the smallest and largest eigenvalues of positive definite matrix $Q$, $\sigma_A$ denotes the smallest eigenvalues of matrix $A A^T$, $c_{a}, c_{\rho}$ is two constants defined by
\begin{equation}\nonumber
    \begin{aligned}
        c_{a} & = \max\left\{ (\frac{1+2L^2}{2} + \frac{20L^2}{\sigma_A}) \frac{2}{\tau}, 1 \right\} ,\\
        c_{\rho} & = \max\{\frac{20L^2 + 2\sigma_A L }{\sigma_A\tau},1 \},
    \end{aligned}
\end{equation}
where $\tau  = \frac{ \phi_{\min }^2 \sigma_A }{40 \phi_{\max }^2}+\frac{\sigma_A }{2}$. Then we have that
     \begin{equation}\label{eq:bound:xkplus-xk}\sum_{k=0}^K \left( \mathbb{E}[ \left\|x_{k+1}-x_k\right\|^2] + \frac{4\sigma_{\min}(H)}{\tau \rho} \mathbb{E}[\|y_{k+1} - y_k\|^2] \right) \leq \frac{4(\mathcal{C}_1 + \psi_1-\psi_*)}{\tau c_{\rho}} K^{-1/3},\end{equation}
     where 
     $
\mathcal{C}_1 =      (\frac{c_{a}}{2} + \frac{10}{\sigma_A}) (\frac{\sigma^2}{c_{a}^2} + \frac{2c_{a}^2\sigma^2}{c_{\rho}^3}),~~\psi_k = \mathbb{E}\left[\mathcal{L}_\rho\left(x_k, y_k, \lambda_k\right)\right].
     $ and $\psi_*$ is a lower bound of $\psi_k$.
\end{lemma}

\begin{proof}
  Plugging \eqref{eq:diff-lambda-bound2} into \eqref{eq:single-descent} yields
   \begin{equation}\label{eq:descent-2}
\begin{aligned}
\mathbb{E}\left[\mathcal{L}_\rho\left(x_{k+1}, y_{k+1}, \lambda_{k+1}\right)\right] \leq & \mathbb{E}\left[\mathcal{L}_\rho\left(x_k, y_k, \lambda_k\right)\right]  + (\frac{\nu_k}{2}+ \frac{5}{\rho \sigma_A}  )\mathbb{E}[\|v_k-\nabla f\left(x_k\right)\|^2] + \frac{5}{\rho \sigma_A}\mathbb{E}[\|v_{k-1}-\nabla f\left(x_{k-1}\right)\|^2] \\
&- \sigma_{\min}(H) \mathbb{E}[\|y_{k+1} - y_k\|^2]-\left(\eta \phi_{\min }+\frac{\sigma_A \rho}{2}-\frac{L}{2} - \frac{1}{2\nu_k} - \frac{5 \eta^2 \phi_{\max }^2}{\rho\sigma_A}\right)\mathbb{E}[ \left\|x_{k+1}-x_k\right\|^2]\\
& +\frac{5\left(L^2+\eta^2 \phi_{\max }^2\right)}{\rho\sigma_A}\left\|x_{k-1}-x_k\right\|^2.
\end{aligned}
\end{equation}
Let us first focus on $\|v_k-\nabla f\left(x_k\right)\|^2$, it follows from Lemma \ref{lem:curr} and \ref{lem:diff-nablaf-v} that 
\begin{equation}\label{eq:sum-v-nablaf}
\begin{aligned}
    \sum_{k=0}^K \mathbb{E}[ \|v_k-\nabla f\left(x_k\right)\|^2 ] &  \leq \frac{\mathbb{E}[\|\varepsilon_0\|^2]}{1-(1-a)^2} + \frac{2a^2\sigma^2}{1-(1-a)^2}K + \frac{2L^2(1-a)^2}{1-(1-a)^2} \sum_{k=0}^{K-1} \mathbb{E}[\|x_k - x_{k-1}\|^2] \\
    & \leq \frac{\sigma^2}{a m} + 2a\sigma^2K + \frac{2L^2}{a} \sum_{k=0}^{K-1} \mathbb{E}[\|x_k - x_{k-1}\|^2],
    \end{aligned}
\end{equation}
where the second inequality uses $1 - (1-a)^2 \geq a$ and $\mathbb{E}[\|\varepsilon_0\|^2] \leq \frac{\sigma^2}{m}$. Since $\nu_k$ is any positive number, we let $\nu_k = c_{a}/\rho
$. 
Summing \eqref{eq:descent-2} over $k=0,1,\ldots, K$ and combining with \eqref{eq:sum-v-nablaf} yields
\begin{equation}\nonumber
    \begin{aligned}
&\mathbb{E}\left[\mathcal{L}_\rho\left(x_{K+1}, y_{K+1}, \lambda_{k+1}\right)\right] \leq  \mathbb{E}\left[\mathcal{L}_\rho\left(x_1, y_1, \lambda_1\right)\right]  + (\frac{c_a}{2\rho}+ \frac{10}{\rho \sigma_A}  )\sum_{k=0}^K \mathbb{E}[\|v_k-\nabla f\left(x_k\right)\|^2] \\
&- \sigma_{\min}(H) \mathbb{E}[\|y_{k+1} - y_k\|^2]-\left(\eta \phi_{\min }+\frac{\sigma_A \rho}{2}-\frac{L}{2} - \frac{1}{2\nu_k} - \frac{5\left(L^2+2\eta^2 \phi_{\max }^2\right)}{\rho\sigma_A}\right)\sum_{k=0}^K\mathbb{E}[ \left\|x_{k+1}-x_k\right\|^2] \\
\leq & \mathbb{E}\left[\mathcal{L}_\rho\left(x_1, y_1, \lambda_1\right)\right]  + (\frac{c_a}{2\rho}+ \frac{10}{\rho \sigma_A}  ) (\frac{\sigma^2}{am} + 2a\sigma^2K )- \sigma_{\min}(H) \sum_{k=0}^K\mathbb{E}[\|y_{k+1} - y_k\|^2] \\
&-\left(\eta \phi_{\min }+\frac{\sigma_A \rho}{2}-\frac{L}{2} - \frac{1}{2c_a}\rho - \frac{5\left(L^2+2\eta^2 \phi_{\max }^2\right)}{\rho\sigma_A} - (\frac{L^2}{c_a}+ \frac{20L^2}{a \rho \sigma_A}  )\right)\sum_{k=0}^K\mathbb{E}[ \left\|x_{k+1}-x_k\right\|^2].
\end{aligned}
\end{equation}
Since $\eta = \frac{\phi_{\min } \rho \sigma_A }{20\phi_{\max}^2}$, we obtain that
\begin{equation}\label{eq:l-bound}
    \begin{aligned}
&\mathbb{E}\left[\mathcal{L}_\rho\left(x_{K+1}, y_{K+1}, \lambda_{K+1}\right)\right] 
\leq  \mathbb{E}\left[\mathcal{L}_\rho\left(x_1, y_1, \lambda_1\right)\right]  + \underbrace{ (\frac{c_a}{2\rho}+ \frac{10}{\rho \sigma_A}  ) (\frac{\sigma^2}{am} + 2a\sigma^2K )}_{\Gamma_1}- \sigma_{\min}(H) \sum_{k=0}^K\mathbb{E}[\|y_{k+1} - y_k\|^2] \\ 
&-\underbrace{\left(\frac{ \phi_{\min }^2 \sigma_A \rho}{40 \phi_{\max }^2}+\frac{\sigma_A \rho}{2}-\frac{L}{2} - \frac{1}{2c_a} \rho - \frac{5L^2}{\rho\sigma_A} - (\frac{ L^2}{c_a}+ \frac{20L^2}{a \rho \sigma_A}  )\right)}_{\Gamma_2}\sum_{k=0}^K\mathbb{E}[ \left\|x_{k+1}-x_k\right\|^2].
\end{aligned}
\end{equation}
For $\Gamma_1$,  combining with $a = c_{a}^2/\rho^2, m = \lceil \rho \rceil$ and $\rho = c_{\rho}K^{1/3}$, we have that
$$
\begin{aligned}
\Gamma_1 &= (\frac{c_{a}}{2\rho} + \frac{10}{\sigma_A \rho})(\frac{\sigma^2}{c_{a}^2}\rho + \frac{2c_{a}^2}{\rho^2}\sigma^2K)\\
 & = (\frac{c_{a}}{2} + \frac{10}{\sigma_A}) (\frac{\sigma^2}{c_{a}^2} + \frac{2c_{a}^2\sigma^2}{c_{\rho}^3})=: \mathcal{C}_1,
\end{aligned}
$$
where the second equality uses $K = \frac{\rho^3}{c_{\rho}^3}$.   For $\Gamma_2$, since $c_{a} = \max\left\{ (\frac{1+2L^2}{2} + \frac{20L^2}{\sigma_A}) \frac{2}{\tau}, 1 \right\}$ and $\tau = \frac{ \phi_{\min }^2 \sigma_A }{40 \phi_{\max }^2}+\frac{\sigma_A }{2}$, one has that
$$
\begin{aligned}
\Gamma_2  & = \tau \rho - \frac{L}{2} -\frac{1}{2c_{a}} \rho - \frac{5L^2}{\sigma_A \rho} - \frac{L^2}{c_{a}} -  \frac{20L^2}{c_{a}^2\sigma_A}\rho\\
& = ( \tau - \frac{1+2L^2}{2c_{a}}  - \frac{20L^2}{c_{a}^2 \sigma_A} ) \rho -  \frac{5L^2}{\sigma_A } \frac{1}{\rho}- \frac{L}{2}\\
& \geq ( \tau - \frac{1+2L^2}{2c_{a}}  - \frac{20L^2}{c_{a} \sigma_A} ) \rho -  \frac{5L^2}{\sigma_A }- \frac{L}{2} \\
& \geq  \frac{\tau}{2} \rho -  \frac{5L^2}{\sigma_A }- \frac{L}{2} \\
& \geq \frac{\tau}{4}\rho.
\end{aligned}
$$
where we use $c_{\rho} =  \max\{\frac{20L^2 + 2\sigma_A L }{\sigma_A\tau},1 \}$ and $\rho>c_{\rho}>1$. Plugging those two term into \eqref{eq:l-bound} yields
\begin{equation}\nonumber
   \frac{\tau}{4}\rho\sum_{k=0}^K \left( \mathbb{E}[ \left\|x_{k+1}-x_k\right\|^2] + \frac{4\sigma_{\min}(H)}{\tau \rho} \mathbb{E}[\|y_{k+1} - y_k\|^2] \right) \leq \mathcal{C}_1 + \psi_1 - \psi_{K+1}.
\end{equation}
Since $\rho = c_{\rho}K^{1/3}$, and from Assumption \ref{assum:low-bound}, there exists a low bound $\psi_*$ of the sequence $\{\psi_k\}$, we give \eqref{eq:bound:xkplus-xk} and complete the proof. 
\end{proof}

By combining with Lemmas \ref{lem:relation-optima-x} and \ref{lem:xk-xk+1}, we give the proof of Theorem \ref{theorem:constant}.

\begin{proof}[Proof of Theorem \ref{theorem:constant}]
By the definition of $\epsilon$-stationary point in Definition \ref{def:kkt}, we have that
\begin{equation}\label{eq:temp2}
     \mathbb{E}\left[ \text{dist}^2(0,\partial L(x_k,y_k,\lambda_k))\right] = \mathbb{E}\left\|A^T \lambda_k-\nabla f\left(x_k\right)\right\|^2 + \mathbb{E}\left\|A x_{k+1}+B y_{k+1}-c\right\|^2 + \mathbb{E}\left[\operatorname{dist}^2\left(B^T \lambda_k, \partial h\left(y_k\right)\right)\right] 
\end{equation}
Now we analyze those three terms respectively. 
It follows from Lemma \ref{lem:relation-optima-x} that
\begin{equation}\nonumber
    \begin{aligned}
        \sum_{k=1}^K \mathbb{E}\left\|A^T \lambda_k-\nabla f\left(x_k\right)\right\|^2 & \leq  3 \sum_{k=1}^K( \mathbb{E}[ \|v_{k-1} - \nabla f(x_{k-1})\|^2 ]  + (L^2 + \eta^2 \phi_{\max }^2) \mathbb{E}[ \|x_k - x_{k-1}\|^2]  ) \\
        & \leq \frac{3\sigma^2}{am} + 6a\sigma^2K + ( \frac{6L^2}{a} + L^2 + \eta^2 \phi_{\max }^2 ) \sum_{k=1}^K \mathbb{E}[\|x_k - x_{k-1}\|^2] \\
        & \leq \frac{3\sigma^2 \rho}{c_{a}^2} + \frac{6c_{a}^2\sigma^2}{\rho^2}K +( \frac{6L^2 \rho^2}{c_{a}^2} +L^2 +  \frac{\phi_{\min}^2\sigma_A^2\rho^2}{400\phi_{\max}^2})\sum_{k=1}^K \mathbb{E}[\|x_k - x_{k-1}\|^2] \\
        & \leq \left(  \frac{3\sigma^2 c_\rho}{c_{a}^2} +    \frac{6c_{a}^2\sigma^2}{c_\rho^2}   +   ( \frac{6L^2 \rho^2}{c_{a}^2} +L^2 +  \frac{\phi_{\min}^2\sigma_A^2\rho^2}{400\phi_{\max}^2}) \frac{4\mathcal{C}_1 + \psi_1-\psi_*}{\tau c_{\rho}}  \right) K^{1/3},
    \end{aligned}
\end{equation}
where the first inequality uses \eqref{eq:bound-alambda}, the second inequality utilizes \eqref{eq:sum-v-nablaf}, the third inequality uses the definition of $a,m$ and $\eta$, the last inequality follows from \eqref{eq:bound:xkplus-xk} and the definition of $\rho$. Now we consider the second term in \eqref{eq:temp2}. It follows from \eqref{eq:bound-AB} that
\begin{equation}\nonumber
    \begin{aligned}
      &  \sum_{k=1}^K \mathbb{E}\left\|A x_{k+1}+B y_{k+1}-c\right\|^2   = \frac{1}{\rho^2} \sum_{k=1}^K \mathbb{E}[\left\|\lambda_{k+1}-\lambda_k\right\|^2] \\
        & \leq \frac{1}{\rho^2} ( \frac{10}{\sigma_A} \sum_{k=1}^K \mathbb{E}[ \|v_k 
 - \nabla f(x_k)\|^2] + \frac{5(L^2 + 2\eta^2 \phi_{\max}^2)}{\sigma_A}  \sum_{k=1}^K \mathbb{E}[\|x_k - x_{k+1}\|^2] )  \\
 & \leq \frac{10\sigma^2}{am \rho^2 \sigma_A} + \frac{20a\sigma^2K}{\rho^2 \sigma_A} + \frac{20L^2}{a \rho^2 \sigma_A} \sum_{k=1}^K \mathbb{E}[\|x_k - x_{k-1}\|^2]  + \frac{5(L^2 + 2\eta^2 \phi_{\max}^2)}{\sigma_A \rho^2}  \sum_{k=1}^K \mathbb{E}[\|x_k - x_{k+1}\|^2] \\
 & \leq \frac{10\sigma^2}{c_{a}^2 \sigma_A c_\rho} K^{-1/3} + \frac{20c_{a}^2\sigma^2}{c_\rho^4 \sigma_A} K^{-1/3} + \frac{80L^2(\mathcal{C}_1 + \psi_1-\psi_*)}{c_{a}^2  \sigma_A \tau c_{\rho}} K^{-1/3} +  \frac{20L^2(\mathcal{C}_1 + \psi_1-\psi_*)}{\tau \sigma_A c_{\rho}^3} K^{-1} + \frac{\phi_{\min}^2 \sigma_A(\mathcal{C}_1 + \psi_1-\psi_*)}{10\phi_{\max}^2\tau c_{\rho}} K^{-1/3},
    \end{aligned}
\end{equation}
where the first inequality uses \eqref{eq:diff-lambda-bound2}, the second inequality utilizes \eqref{eq:sum-v-nablaf}. 
Finally, we focus on the last term in \eqref{eq:temp2}. It follows from \eqref{eq:bound-partial} that
\begin{equation}\nonumber
\begin{aligned}
    \sum_{k=1}^K   \mathbb{E}\left[\operatorname{dist}\left(B^T \lambda_k, \partial h\left(y_k\right)\right)\right]^2  \leq &2\rho^2\|B\|_2^2\|A\|_2^2 \sum_{k=1}^K \left(  \mathbb{E}[\left\|x_k-x_{k-1}\right\|^2] + \frac{\sigma^2_{\max}(H)}{  \rho^2\|B\|_2^2\|A\|_2^2} \mathbb{E}[\|y_k - y_{k-1}\|^2 ]\right) \\
    \leq &2\rho^2\|B\|_2^2\|A\|_2^2 \sum_{k=1}^K \left(  \mathbb{E}[\left\|x_k-x_{k-1}\right\|^2] + \frac{4\sigma_{\min}(H)}{\tau \rho} \mathbb{E}[\|y_k - y_{k-1}\|^2] \right) \\
    \leq &  \frac{2c_\rho\|B\|_2^2\|A\|_2^2(\mathcal{C}_1 \psi_1)}{\tau} K^{1/3}.
    \end{aligned}
\end{equation}
where the second inequality uses $\rho \geq c_{\rho} \geq \frac{\tau \sigma^2_{\max}(H)}{4 \|A\|^2 \|B\|^2 \sigma_{\min}(H)}$. 
Let us denote 
$$
\begin{aligned}
\mathcal{H}_1 : &= \frac{3\sigma^2 c_\rho}{c_{a}^2} +    \frac{6c_{a}^2\sigma^2}{c_\rho^2}   +   ( \frac{6L^2 \rho^2}{c_{a}^2} +L^2 +  \frac{\phi_{\min}^2\sigma_A^2\rho^2}{400\phi_{\max}^2}) \frac{4\mathcal{C}_1 + \psi_1-\psi_*}{\tau c_{\rho}}  + \frac{2c_\rho\|B\|_2^2\|A\|_2^2(\mathcal{C}_1 \psi_1)}{\tau}, \\
\mathcal{H}_2: &= \frac{10\sigma^2}{c_{a}^2 \sigma_A c_\rho}  + \frac{20c_{a}^2\sigma^2}{c_\rho^4 \sigma_A}  + \frac{80L^2(\mathcal{C}_1 + \psi_1-\psi_*)}{c_{a}^2  \sigma_A \tau c_{\rho}}  + \frac{\phi_{\min}^2 \sigma_A(\mathcal{C}_1 + \psi_1-\psi_*)}{10\phi_{\max}^2\tau c_{\rho}},\\
\mathcal{H}_3: &=    \frac{20L^2(\mathcal{C}_1 + \psi_1-\psi_*)}{\tau \sigma_A c_{\rho}^3}.
\end{aligned}
$$
Then we have that 
\begin{equation}\nonumber
\begin{aligned}
    \min_{1\leq k \leq K} \mathbb{E}\left[ \text{dist}^2(0,\partial L(x_k,y_k,\lambda_k))\right] &  \leq \frac{1}{K} \sum_{k=1}^K \mathbb{E}\left[ \text{dist}^2(0,\partial L(x_k,y_k,\lambda_k))\right] \\
    & \leq  \frac{1}{K}\left( \mathcal{H}_1 K^{1/3} +  \mathcal{H}_2 K^{-1/3} + \mathcal{H}_3 K^{-1}  \right) \\
    & \leq  \mathcal{H}_1 K^{-2/3} +  \mathcal{H}_2 K^{-4/3} + \mathcal{H}_3 K^{-2}. 
    \end{aligned}
\end{equation}
The proof is completed.

\end{proof}

\subsection{Proof of Section \ref{sec:diminish}}\label{appendix-diminish}
 Let us first define the Lyapunov function as follows:
\begin{equation}\label{def:Phi}
    \Phi_k: = \mathcal{L}_\rho\left(x_k, y_k, \lambda_k\right) + \gamma_k \|\varepsilon_k\|^2,
\end{equation}
where $\epsilon_k = v_k - \nabla f(x_k)$ and $\gamma_k >0$ is a parameter, that will be given in next. Before providing the proof of Theorem \ref{theorem:diminish}, we establish the following descent lemma.
\begin{lemma}
Suppose that Assumptions \ref{assm:lipsciz}-\ref{assum:variance} hold.     Let the sequence $\left\{x_k, y_k, \lambda_k\right\}_{k=1}^K$ be generated by Algorithm \ref{alg:sam}. Assume that $\rho_k = c_{\rho}k^{1/3}$. Then
     \begin{equation}\label{eq:Phi-minus}
    \begin{aligned}
        \Phi_{k+1} - \Phi_k   \leq & (\frac{\nu_k}{2}+ \frac{10}{{\rho_k} \sigma_A} + \gamma_{k+1}(1-a_{k+1})^2 - \gamma_k )\|\varepsilon_k\|^2 +  \frac{10}{{\rho_k} \sigma_A}  \|\varepsilon_{k-1}\|^2 - \sigma_{\min}(H) \mathbb{E}[\|y_{k+1} - y_k\|^2]\\
       &-\left(\eta_k \phi_{\min }+\frac{\sigma_A {\rho_k}}{2}-\frac{L}{2} - \frac{1}{2\nu_k} - \frac{10 \eta_k^2 \phi_{\max }^2}{{\rho_k}\sigma_A} - 2\gamma_{k+1}L^2(1-a_{k+1})^2 \right)\mathbb{E}[ \left\|x_{k+1}-x_k\right\|^2] \\
       & +\frac{10\left(L^2+\eta_k^2 \phi_{\max }^2\right)}{{\rho_k}\sigma_A} \left\|x_{k-1}-x_k\right\|^2  + 2\gamma_{k+1}a_{k+1}^2  \sigma^2.
    \end{aligned}
\end{equation}
\end{lemma}

\begin{proof}
Since $\rho_k = c_{\rho}k^{1/3}$, one has that $\rho_{k+1} - \rho_k \leq 2\rho_k$, and \eqref{eq:single-descent} reduced to
 \begin{equation}\label{eq:single-descent-1}
\begin{aligned}
\mathbb{E}\left[\mathcal{L}_{\rho_{k+1}}\left(x_{k+1}, y_{k+1}, \lambda_{k+1}\right)\right] \leq & \mathbb{E}\left[\mathcal{L}_{\rho_k}\left(x_k, y_k, \lambda_k\right)\right]+\frac{2}{{\rho_k}} \mathbb{E} [\left\|\lambda_{k+1}-\lambda_k\right\|^2] + \frac{\nu_k}{2}\mathbb{E}[\|v_k-\nabla f\left(x_k\right)\|^2] \\
&- \sigma_{\min}(H) \mathbb{E}[\|y_{k+1} - y_k\|^2]-\left(\eta_k \phi_{\min }+\frac{\sigma_A {\rho_k}}{2}-\frac{L}{2} - \frac{1}{2\nu_k}\right)\mathbb{E}[ \left\|x_{k+1}-x_k\right\|^2],
\end{aligned}
\end{equation}
     Plugging \eqref{eq:diff-lambda-bound2} into \eqref{eq:single-descent-1} yields
   \begin{equation}\label{eq:descent-3}
\begin{aligned}
\mathbb{E}\left[\mathcal{L}_{\rho_{k+1}}\left(x_{k+1}, y_{k+1}, \lambda_{k+1}\right)\right] \leq & \mathbb{E}\left[\mathcal{L}_{\rho_k}\left(x_k, y_k, \lambda_k\right)\right]  + (\frac{\nu_k}{2}+ \frac{10}{{\rho_k} \sigma_A}  )\mathbb{E}[\|v_k-\nabla f\left(x_k\right)\|^2] + \frac{10}{{\rho_k} \sigma_A}\mathbb{E}[\|v_{k-1}-\nabla f\left(x_{k-1}\right)\|^2] \\
&- \sigma_{\min}(H) \mathbb{E}[\|y_{k+1} - y_k\|^2]-\left(\eta_k \phi_{\min }+\frac{\sigma_A {\rho_k}}{2}-\frac{L}{2} - \frac{1}{2\nu_k} - \frac{10 \eta_k^2 \phi_{\max }^2}{{\rho_k}\sigma_A}\right)\mathbb{E}[ \left\|x_{k+1}-x_k\right\|^2]\\
& +\frac{10\left(L^2+\eta_k^2 \phi_{\max }^2\right)}{{\rho_k}\sigma_A}\left\|x_{k-1}-x_k\right\|^2.
\end{aligned}
\end{equation}

Let us denote $\psi_k: = \mathcal{L}_{\rho_k}\left(x_k, y_k, \lambda_k\right)$ and $\varepsilon_k = v_k -\nabla f(x_k)$.  Then
\begin{equation}\nonumber
    \begin{aligned}
       &   \psi_{k+1} - \psi_k + \gamma_{k+1} \|\varepsilon_{k+1}\|^2 \\
       \leq &  (\frac{\nu_k}{2}+ \frac{10}{{\rho_k} \sigma_A}  )\|\varepsilon_k\|^2 + \frac{10}{{\rho_k} \sigma_A}\|\varepsilon_{k-1}\|^2+\frac{10\left(L^2+\eta_k^2 \phi_{\max }^2\right)}{{\rho_k}\sigma_A}\left\|x_{k-1}-x_k\right\|^2 \\
       &- \sigma_{\min}(H) \mathbb{E}[\|y_{k+1} - y_k\|^2]-\left(\eta_k \phi_{\min }+\frac{\sigma_A {\rho_k}}{2}-\frac{L}{2} - \frac{1}{2\nu_k} - \frac{10 \eta_k^2 \phi_{\max }^2}{{\rho_k}\sigma_A}\right)\mathbb{E}[ \left\|x_{k+1}-x_k\right\|^2] \\
       & + \gamma_{k+1}\left( 
  \left(1-a_{k+1}\right)^2 \mathbb{E}\left\|\varepsilon_{k}\right\|^2+2a_{k+1}^2 \sigma^2+2 L^2\left(1-a_{k+1}\right)^2 \mathbb{E}\left\|x_{k+1}-x_{k}\right\|^2\right)\\
       \leq & (\frac{\nu_k}{2}+ \frac{10}{{\rho_k} \sigma_A} + \gamma_{k+1}(1-a_{k+1})^2 )\|\varepsilon_k\|^2 +  \frac{10}{{\rho_k} \sigma_A}  \|\varepsilon_{k-1}\|^2 \\
       &-\left(\eta_k \phi_{\min }+\frac{\sigma_A {\rho_k}}{2}-\frac{L}{2} - \frac{1}{2\nu_k} - \frac{10 \eta_k^2 \phi_{\max }^2}{{\rho_k}\sigma_A} - 2\gamma_{k+1}L^2(1-a_{k+1})^2 \right)\mathbb{E}[ \left\|x_{k+1}-x_k\right\|^2] \\
       & +\frac{10\left(L^2+\eta_k^2 \phi_{\max }^2\right)}{{\rho_k}\sigma_A} \left\|x_{k-1}-x_k\right\|^2  + 2\gamma_{k+1}a_{k+1}^2  \sigma^2.
    \end{aligned}
\end{equation}
The proof is completed.
\end{proof}

\begin{theorem}\label{theo:temp}
   Under the same setting in Theorem \ref{theorem:diminish},  let the sequence $\left\{x_k, y_k, \lambda_k\right\}_{k=1}^K$ be generated by Algorithm \ref{alg:sam}. Assume that 
   $$\rho_k = c_{\rho}k^{1/3}, a_{k+1} = c_ak^{-2/3}, \eta_k = c_{\eta}k^{1/3}, \nu_k = c_{\nu}/\rho_k,\gamma_{k+1} = c_{\gamma}k^{1/3},$$
   where $c_{\nu}, c_{\gamma},c_{\rho},c_a, c_{\eta}$ satisfy that
   \begin{equation}\nonumber
    \begin{aligned}
        c_{\nu} & \geq \frac{1}{4\sigma_A},~~c_{\gamma} \leq \frac{\sigma_A c_{\rho}}{16L^2},~c_{\eta} \leq \frac{\sigma_A c_{\rho}}{\sqrt{160}\phi_{\max}}\\
       c_{\rho} & \geq \frac{8L}{\sigma_A} + \frac{160L^2}{\sigma_A^2} + \frac{\|A\|\|B\|}{\sigma^2_{\max}(H)},\\
        c_a & \geq \frac{3c_{\nu}c_{\rho} + 60 + 2c_{\gamma}\sigma_A c_{\rho}}{3c_{\gamma}\sigma_A c_{\rho}}.
    \end{aligned}
\end{equation}
 Then we have that
\begin{equation}\nonumber
    \sum_{k=1}^K k^{-1/3} \mathbb{E}[\| \epsilon_k  \|^2] +  \sum_{k=1}^K k^{1/3} \mathbb{E}[\|x_{k+1} - x_k\|^2]+\sum_{k=1}^K \mathbb{E}[\|y_{k+1} - y_k\|^2]  \leq \frac{\Phi_1 - \Phi_* + 2\sigma^2 c_a^2 c_{\gamma} \ln(K)}{\min(\mathcal{C}_3,\mathcal{C}_4,\sigma_{\min}(H))}.
\end{equation}
\end{theorem}

\begin{proof}
Telescoping \eqref{eq:Phi-minus} from $k=1,\cdots,K$ gives
      \begin{equation}\label{eq:bound-phi}
    \begin{aligned}
        \Phi_{K+1} - \Phi_1   \leq & \sum_{k=1}^K  \underbrace{ (\frac{\nu_k}{2}+ \frac{10}{{\rho_k} \sigma_A} + \gamma_{k+1}(1-a_{k+1})^2 - \gamma_k )}_{\Gamma_3} \mathbb{E}[\|\varepsilon_k\|^2]  +  2\sigma^2 \sum_{k=1}^K\gamma_{k+1}a_{k+1}^2  \sigma^2- \sigma_{\min}(H) \sum_{k=1}^K \mathbb{E}[\|y_{k+1} - y_k\|^2] \\
       &-  \sum_{k=1}^K\underbrace{ \left(\eta_k \phi_{\min }+\frac{\sigma_A {\rho_k}}{2}-\frac{L}{2} - \frac{1}{2\nu_k} - \frac{10\left(L^2+2\eta_k^2 \phi_{\max }^2\right)}{{\rho_k}\sigma_A} - 2\gamma_{k+1}L^2(1-a_{k+1})^2 \right)}_{\Gamma_4}\mathbb{E}[ \left\|x_{k+1}-x_k\right\|^2].
    \end{aligned}
\end{equation}
Next, we bound the terms $\Gamma_3$ and $\Gamma_4$, respectively.    Since $\nu_k = c_{\nu}/{\rho_k}$ and $\rho_k = c_{\rho} k^{1/3}$,  one has that
    \begin{equation}\nonumber
        \begin{aligned}
\Gamma_3 &= \left( \frac{c_{\nu}}{2} + \frac{10}{\sigma_A} \right ) \frac{1}{{\rho_k}}  + \gamma_{k+1}(1-a_{k+1})^2 - \gamma_k  \\
           &\leq  \left( \frac{c_{\nu}}{2c_{\rho}} + \frac{10}{\sigma_A c_{\rho}} \right ) k^{-1/3}   + \gamma_{k+1} - \gamma_k - a_{k+1} \gamma_{k+1},
        \end{aligned}
    \end{equation}
   where the last inequality uses $(1-a_{k+1})^2 \leq (1-a_{k+1})$. Consider the convex function $l(x): = x^{1/3}$. By first order characterization, $l(x+1)\leq l(x) + l^{\prime}(x) = x^{1/3} + \frac{1}{3}x^{-2/3}$. Since $\gamma_{k+1} = c_{\gamma}k^{1/3}$, we have $\gamma_{k+1} - \gamma_k \leq   \frac{c_{\gamma}}{3} k^{-2/3}$. Combining with $a_{k+1} = c_ak^{-2/3}$ yields
   \begin{equation}\nonumber
       \begin{aligned}
           \Gamma_3 
           & \leq  \left( \frac{c_{\nu}}{2c_{\rho}} + \frac{10}{\sigma_A c_{\rho}} \right ) k^{-1/3}   + \frac{c_{\gamma}}{3} k^{-1/3} -  c_a c_\gamma k^{-1/3}  \\
                 & \leq   \frac{3c_{\nu} c_{\rho} + 60 + 2c_{\gamma}\sigma_A c_{\rho}}{6\sigma_A c_{\rho}} k^{-1/3}  -  c_a c_\gamma k^{-1/3}  \\
           & \leq - \frac{c_a c_{\gamma}}{2}  k^{-1/3},
       \end{aligned}
   \end{equation}
   where the first inequality uses the fact that $k^{-2/3} \leq k^{-1/3}$, the last inequality uses 
   $$
   c_a \geq \frac{3c_{\nu}c_{\rho} + 60 + 2c_{\gamma}\sigma_A c_{\rho}}{3c_{\gamma}\sigma_A c_{\rho}}.
   $$
   For $\Gamma_4$, since $\eta_k = c_{\eta} k^{1/3}$,  we have that
   \begin{equation}\nonumber
       \begin{aligned}
           \Gamma_4 & = c_\eta k^{1/3} \phi_{\min }+\frac{\sigma_A c_\rho}{2}k^{1/3}-\frac{L}{2} - \frac{c_{\rho}}{2c_{\nu}} k^{ 1/3} -\frac{10\left(L^2+2c_\eta^2 k^{2/3} \phi_{\max }^2\right)}{c_\rho\sigma_A k^{1/3}} - 2c_\gamma L^2 k^{1/3} \\
           & \geq \left( \frac{\sigma_A c_\rho}{2} - \frac{c_{\rho}}{2c_{\nu}} - \frac{20 c_\eta^2  \phi_{\max }^2}{c_\rho\sigma_A } -2c_\gamma L^2 \right)k^{1/3} - \frac{L}{2} - \frac{10L^2}{c_{\rho}\sigma_A}\\
              & \geq \left( \frac{\sigma_A c_\rho}{2} -\frac{\sigma_A c_\rho}{8}-\frac{\sigma_A c_\rho}{8}-\frac{\sigma_A c_\rho}{8}  \right)k^{1/3} - \frac{L}{2} - \frac{10L^2}{c_{\rho}\sigma_A}\\
           & \geq (\frac{\sigma_A c_\rho}{8}   - \frac{L}{2} - \frac{10L^2}{c_{\rho}\sigma_A}) k^{1/3}\\
           & \geq  \frac{\sigma_A c_\rho}{16} k^{1/3}
       \end{aligned}
   \end{equation}
   where the first inequality uses 
$1\leq k \leq K$, the second inequality utilizes $c_{\nu}\geq \frac{1}{4\sigma_A}$ and $c_{\eta} \leq \frac{\sigma_A c_{\rho}}{\sqrt{160}\phi_{\max}}$, $c_{\gamma} \leq \frac{\sigma_A c_{\rho}}{16L^2}$. The last inequality use $c_{\rho}\geq \frac{8L}{\sigma_A} + \frac{160L^2}{\sigma_A^2}$. Let us denote $\mathcal{C}_3 = \frac{c_ac_{\gamma}}{2}$ and $\mathcal{C}_4 = \frac{\sigma_A c_\rho}{16}$. It follows from Assumption \ref{assum:low-bound} that there exists a low bound $\Phi_*$ for the sequence $\{\Phi_k\}$. 
   Plugging those term into \eqref{eq:bound-phi} yields
\begin{equation}\nonumber
\begin{aligned}
  & \mathcal{C}_3 \sum_{k=1}^K k^{-1/3} \mathbb{E}[\| \epsilon_k  \|^2] + \mathcal{C}_4 k^{1/3}\sum_{k=1}^K  \mathbb{E}[\|x_{k+1} - x_k\|^2] + \sigma_{\min}(H) \sum_{k=1}^K \mathbb{E}[\|y_{k+1} - y_k\|^2] \\
  \leq & \Phi_1 - \Phi_* + 2\sigma^2 c_a^2 c_{\gamma} \sum_{k=1}^K k^{-1} \leq \Phi_1 - \Phi_* + 2\sigma^2 c_a^2 c_{\gamma} \ln(K).
  \end{aligned}
\end{equation}
% \end{proof}
The proof is completed.
\end{proof}

Now we provide the proof of Theorem \ref{theorem:diminish}. 
\begin{proof}[Proof of Theorem \ref{theorem:diminish}]
It follows from Lemma \ref{lem:relation-optima-x} and Theorem  \ref{theo:temp} that
    \begin{equation}\nonumber
        \begin{aligned}
              \sum_{k=1}^K \mathbb{E}\left\|A^T \lambda_k-\nabla f\left(x_k\right)\right\|^2 & \leq  3 \sum_{k=1}^K( \mathbb{E}[ \|v_{k-1} - \nabla f(x_{k-1})\|^2 ]  + (L^2 + \eta_k^2 \phi_{\max }^2) \mathbb{E}[ \|x_k - x_{k-1}\|^2]  )\\ 
              & \leq  6 \sum_{k=1}^K( \mathbb{E}[ \|v_{k-1} - \nabla f(x_{k-1})\|^2 ]  +  \phi_{\max }^2\eta_k^2  \mathbb{E}[ \|x_k - x_{k-1}\|^2]  ) \\
              & \leq  6 c_{\gamma}^2 \phi_{\max}^2 \sum_{k=1}^K( \mathbb{E}[ \|v_{k-1} - \nabla f(x_{k-1})\|^2 ]  +  k^{2/3}  \mathbb{E}[ \|x_k - x_{k-1}\|^2]  ) \\
              & \leq  6 c_{\gamma}^2 \phi_{\max}^2  \sum_{k=1}^K k^{1/3}( k^{-1/3} \mathbb{E}[ \|v_{k-1} - \nabla f(x_{k-1})\|^2 ]  +  k^{1/3}  \mathbb{E}[ \|x_k - x_{k-1}\|^2]  ) \\
              & \leq 6 c_{\gamma}^2 \phi_{\max}^2  \frac{\Phi_1 -\Phi_*+ 2\sigma^2 c_a^2 c_{\gamma} \ln(K)}{\min(\mathcal{C}_3,\mathcal{C}_4,\sigma_{\min}(H))} K^{1/3},
        \end{aligned}
    \end{equation}
    and 
    \begin{equation}\nonumber
    \begin{aligned}
      &  \sum_{k=1}^K \mathbb{E}\left\|A x_{k+1}+B y_{k+1}-c\right\|^2   =  \sum_{k=1}^K \frac{1}{\rho_k^2}\mathbb{E}[\left\|\lambda_{k+1}-\lambda_k\right\|^2] \\
        &  \stackrel{\eqref{eq:diff-lambda-bound2}}{\leq}   \frac{10}{\sigma_A} \sum_{k=1}^K \frac{1}{\rho_k^2}\mathbb{E}[\|v_k 
 - \nabla f(x_k)\|^2] + \frac{5(L^2 + 2 \phi_{\max}^2)}{\sigma_A}  \sum_{k=1}^K \frac{\eta_k^2}{\rho_k^2}\mathbb{E}[\|x_k - x_{k+1}\|^2]   \\
 %& \leq \frac{1}{\rho_k^2} ( \frac{10}{\sigma_A} \sum_{k=1}^K \mathbb{E}[\|v_k - \nabla f(x_k)\|^2] + \frac{5(L^2 + 2 \phi_{\max}^2)}{\sigma_A}\eta_k^2  \sum_{k=1}^K \mathbb{E}[\|x_k - x_{k+1}\|^2] ) \\
 & \leq   \frac{10}{c_{\rho}^2\sigma_A} \sum_{k=1}^K k^{-2/3} \mathbb{E}[\|v_k 
 - \nabla f(x_k)\|^2] + \frac{5(L^2 + 2 \phi_{\max}^2)c_{\eta}^2}{\sigma_A c_{\rho}^2}  \sum_{k=1}^K \mathbb{E}[\|x_k - x_{k+1}\|^2]  \\
   & \leq   \frac{10+5(L^2 + 2 \phi_{\max}^2)c_{\eta}^2}{c_{\rho}^2\sigma_A} \sum_{k=1}^K k^{-1/3} \left( k^{-1/3} \mathbb{E}[\|v_k 
  - \nabla f(x_k)\|^2] + k^{1/3} \mathbb{E}[\|x_k - x_{k+1}\|^2] \right)  \\
  & \leq   \frac{10+5(L^2 + 2 \phi_{\max}^2)c_{\eta}^2}{c_{\rho}^2\sigma_A} \sum_{k=1}^K  \left( k^{-1/3} \mathbb{E}[\|v_k 
 - \nabla f(x_k)\|^2] + k^{1/3} \mathbb{E}[\|x_k - x_{k+1}\|^2] \right)  \\
 & \leq  \frac{10+5(L^2 + 2 \phi_{\max}^2)c_{\eta}^2}{c_{\rho}^2\sigma_A} \frac{\Phi_1 -\Phi_*+ 2\sigma^2 c_a^2 c_{\gamma} \ln(K)}{\min(\mathcal{C}_3,\mathcal{C}_4,\sigma_{\min}(H))},
    \end{aligned} 
\end{equation}
where the second inequality use $\eta_k>1$ since $c_{\eta}>1$, the last inequality follows from Theorem \ref{theo:temp}.
Finally, 
\begin{equation}\nonumber
\begin{aligned}
    \sum_{k=1}^K   \mathbb{E}\left[\operatorname{dist}\left(B^T \lambda_k, \partial h\left(y_k\right)\right)\right]^2  &\leq 2\|B\|_2^2\|A\|_2^2 \sum_{k=1}^K \rho_k \left( \rho_k  \mathbb{E}[\left\|x_k-x_{k-1}\right\|^2] + \frac{\sigma^2_{\max}(H)}{  \rho_k\|B\|_2^2\|A\|_2^2} \|y_k - y_{k-1}\|^2 \right) \\
    &\leq 2c_{\rho}K^{1/3}\|B\|_2^2\|A\|_2^2 \sum_{k=1}^K \left( k^{1/3}  \mathbb{E}[\left\|x_k-x_{k-1}\right\|^2] + \frac{\sigma^2_{\max}(H)}{  c_{\rho}^2\|B\|_2^2\|A\|_2^2} \mathbb{E}[\|y_k - y_{k-1}\|^2] \right) \\
      &\leq 2c_{\rho}K^{1/3}\|B\|_2^2\|A\|_2^2 \sum_{k=1}^K \left( k^{1/3}  \mathbb{E}[\left\|x_k-x_{k-1}\right\|^2] + \mathbb{E}[\|y_k - y_{k-1}\|^2] \right) \\
   &\leq 2c_{\rho}\|B\|_2^2\|A\|_2^2 \frac{\Phi_1 -\Phi_*+ 2\sigma^2 c_a^2 c_{\gamma} \ln(K)}{\min(\mathcal{C}_3,\mathcal{C}_4,\sigma_{\min}(H))} K^{1/3},
   \end{aligned}
\end{equation}
where the first inequality uses \eqref{eq:bound-partial},  the third inequality uses $c_{\rho} \geq \frac{\|A\|\|B\|}{\sigma^2_{\max}(H)}$. 
Let us denote 
$$
\begin{aligned}
\mathcal{G}_1 : &= 6 c_{\gamma}^2 \phi_{\max}^2  \frac{\Phi_1 -\Phi_*+ 2\sigma^2 c_a^2 c_{\gamma} \ln(K)}{\min(\mathcal{C}_3,\mathcal{C}_4,\sigma_{\min}(H))}, \\
\mathcal{G}_2: &= \frac{10+ 5(L^2 + 2 \phi_{\max}^2)c_{\eta}^2}{c_{\rho}^2 \sigma_A } \frac{\Phi_1 -\Phi_*+ 2\sigma^2 c_a^2 c_{\gamma} \ln(K)}{\min(\mathcal{C}_3,\mathcal{C}_4,\sigma_{\min}(H))},\\
\mathcal{G}_3: &=    2c_\rho^2\|B\|_2^2\|A\|_2^2 \frac{\Phi_1 -\Phi_*+ 2\sigma^2 c_a^2 c_{\gamma} \ln(K)}{\min(\mathcal{C}_3,\mathcal{C}_4,\sigma_{\min}(H))}.
\end{aligned}
$$
Then we have that 
\begin{equation}\nonumber
\begin{aligned}
    \min_{1\leq k \leq K} \mathbb{E}\left[ \text{dist}^2(0,\partial L(x_k,y_k,\lambda_k))\right] &  \leq \frac{1}{K} \sum_{k=1}^K \mathbb{E}\left[ \text{dist}^2(0,\partial L(x_k,y_k,\lambda_k))\right] \\
    & \leq  \frac{1}{K}\left( \mathcal{H}_1 K^{1/3} +  \mathcal{H}_2  + \mathcal{H}_3 K^{1/3}  \right) \\
    & \leq  \mathcal{G}_1 K^{-2/3} +  \mathcal{G}_2 K^{-1} + \mathcal{G}_3 K^{-2/3}. 
    \end{aligned}
\end{equation}
The proof is completed.
\end{proof}

\section{Application to Plug-and-Play algorithm}

The PnP approach is a versatile methodology primarily utilized for addressing inverse problems involving large-scale measurements through the integration of statistical priors defined as denoisers. This approach draws inspiration from well-established proximal algorithms commonly employed in nonsmooth composite optimization, such as the proximal gradient algorithm, Douglas-Rachard splitting algorithm, and ADMM algorithm, etc. The regularization inverse problem can be written as  
\begin{equation}\nonumber
    \min_x \mathbb{E}[ f(x,\xi)] + h(x).
\end{equation}
This is corresponding to an instance of problem \eqref{prob} by letting $A = I, B = -I, c = 0$. 
Recall that the update rule of $y_{k+1}$ in  Algorithm \ref{alg:sam} can be represented as a proximal operator when $H = r I - \rho B^\top B = (r-\rho)I$:
\begin{equation}
   y_{k+1} =\mathrm{prox}_{h/r} \left(  \frac{r-\rho}{r} y_k + \frac{\rho}{r} (x_k  - \lambda_k/\rho) \right).
\end{equation}
We propose a PnP-SMADMM by replacing the proximal operator with a denoiser operator $D_{\theta}$:
$$
y_{k+1} = D_{\theta} \left(  \frac{r-\rho}{r} y_k + \frac{\rho}{r} (x_k  - \lambda_k/\rho) \right),
$$
where $D_{\theta}$ is denoiser operator parameterized by
a neural network with parameters $\theta$. Moreover, we simplify the update rule of $x$-subproblem by considering \eqref{eq:x-simply-update}. The detailed algorithm is referred to as Algorithm \ref{alg:pnp}.
\begin{algorithm}[tb]
\caption{PnP-SMADMM}
\label{alg:pnp}
\textbf{Input}: Parametes $a_k,  \eta_k, m, \rho, H, Q$; initial points $x_0$, $y_0,z_0$.\\
\begin{algorithmic}[1] %[1] enables line numbers
\STATE Sample $\{\xi_{0,t}\}_{t=0}^m$ and let $v_0 = \frac{1}{m} \sum_{t=1}^{m} \nabla f(x_0,\xi_{0,t})$.
\FOR{$k = 0,\cdots,K-1$}
\STATE $ y_{k+1} = D_{\theta} (x_k - \lambda_k/\rho)$.
\STATE $    x_{k+1} = x_k - \frac{1}{\eta_k}(v_k + \rho(x_k  - y_{k+1} - \frac{\lambda_k}{\rho})).$
\STATE $      \lambda_{k+1} = \lambda_k - \rho (x_{k+1} - y_{k+1}).$
\STATE Sample $\xi_{k+1}\in \mathcal{D}$ and let
\begingroup
\fontsize{9pt}{10pt}\selectfont
$$v_{k+1} = \nabla f(x_{k+1},\xi_{k+1}) + (1-a_{k+1})(v_{k} - \nabla f(x_{k},\xi_{k+1})).$$
\endgroup
\ENDFOR
\end{algorithmic}
%\textbf{Output}: $(x_K,y_K,\lambda_K)$.
\end{algorithm}

To guarantee the theoretical convergence, we consider the gradient step (GS) denoiser developed in \cite{cohen2021has,hurault2021gradient,hurault2022proximal} as follows:
\begin{equation}\label{denoiser}
D_\theta=I-\nabla g_\theta,
\end{equation}
which is obtained from a scalar function
$
g_\theta=\frac{1}{2}\left\|x-N_\theta(x)\right\|^2,
$
where the mapping $N_\theta(\mathbf{x})$ is realized as a differentiable neural network, enabling the explicit computation of $g_\theta$ and ensuring that $g_\theta$ has a Lipschitz gradient with a constant $L_g<1$. Originally, the denoiser $D_\theta$ in \eqref{denoiser} is trained to denoise images degraded with Gaussian noise of level $\theta$. In \cite{hurault2021gradient}, it is shown that, although constrained to be an exact conservative field, it can realize state-of-the-art denoising. Remarkably, the denoiser $D_\theta$ in \eqref{denoiser} takes the form of a proximal mapping of a weakly convex function, as stated in the next proposition.

\begin{proposition}[\cite{hurault2022proximal}, Propostion 3.1]\label{prop:ded} $D_\theta(x)=\operatorname{prox}_{\phi_\theta}(x)$, where $\phi_\theta$ is defined by
\begin{equation}\label{thet-phi}
\phi_\theta(x)=g_\theta\left(D_\theta^{-1}(x)\right)-\frac{1}{2}\left\|D_\theta^{-1}(x)-x\right\|^2
\end{equation}
if $x \in \operatorname{Im}\left(D_\theta\right)$ and $\phi_\theta(x)=+\infty$, otherwise. Moreover, $\phi_\theta$ is $\frac{L_g}{L_g+1}$-weakly convex and $\nabla \phi_\theta$ is $\frac{L_g}{1-L_g}$-Lipschitz on $\operatorname{Im}\left(D_\theta\right)$, and $\phi_\theta(x) \geq g_\theta(x) \forall x \in \mathbb{R}^n$.
\end{proposition}
Drawing upon Proposition \ref{prop:ded}, we are interested in developing the PnP-SMADMM algorithm with a plugged denoiser $D_\theta$ in \eqref{denoiser} that corresponds to the proximal operator of a weakly function $\phi_\theta$ in \eqref{thet-phi}. To do so, we turn to target the optimization problems as follows:
\begin{equation}\label{prob:pnp}
\min F_{r, \theta}(x)=\mathbb{E}_{\xi\in \mathcal{D}}[f(x,\xi)]+r \phi_\theta(x),
\end{equation}
where  $\phi_\theta$ is defined as in Proposition \ref{prop:ded} from the function $g_\theta$ satisfying $D_\theta=I-\nabla g_\theta$. Since $\frac{L_g}{L_g+1} < 1$, the proximal operator is well-defined and we can still apply Theorem \ref{theorem:constant} though the function $\phi_\theta$ is  weakly convex. We give the following convergence result for Algorithm \ref{alg:pnp}.  The proof follows from Theorem \ref{theorem:constant} and we omit it. 
\begin{proposition}
    Under the same conditions as in Theorem \ref{theorem:constant}, let the sequence $\left\{x_k, y_k, \lambda_k\right\}_{k=1}^K$ be generated by Algorithm \ref{alg:pnp}. We assume $L_g<1$. Algorithm \ref{alg:sam} obtains an $\epsilon$-stationary point of \eqref{prob:pnp} with at most $\mathcal{O}(\epsilon^{-\frac{3}{2}})$.
\end{proposition}

\section{Experiments}
In this section, we will compare our algorithm  SMADMM with the existing stochastic ADMM algorithms \cite{huang2016stochastic,zheng2016fast,huang2019faster,zeng2024accelerated} on the Graph-guided binary classification problem. We also compare RED \cite{romano2017little}, PnP-SADMM \cite{sun2021scalable}, SPIDER-ADMM, and ASVRG-ADMM with PnP prior on CT image reconstruction and {nonconvex phase retrieval problems}. 

\subsection{Graph-guided binary Classification}
At the outset, we focus on a binary classification instance that incorporates the correlations among features. Assume that we possess a set of training samples denoted as $\{(a_i,b_i)\}_{i=1}^n$. Here, $a_i$  is an $m$-dimensional vector, and  $b_i$ represents the corresponding label which can only take on the values of either  $-1$  or $+1$. To address this problem, we adopt a model called the graph-guided fused lasso \cite{Kim2009Multivariate}, which demands minimizing the subsequent expression:
\begin{equation}\nonumber
    \min_{x} \frac{1}{N} \sum_{i = 1}^N f_{i}(x)+ \lambda_1 \|Ax\|_1.
\end{equation}
In this context,  $f_i (x)=\frac{1}{1+\exp{(b_i a_i^T x)}}$ symbolizes a sigmoid loss function {which is nonconvex and smooth\cite{Bian_2021}}. The matrix  $A$ is formulated as  $A = [G;I] $, where  $G$ is obtained through sparse inverse covariance matrix estimation as detailed in \cite{Kim2009Multivariate,Friedman2008Sparse}.
Regarding this experiment, we establish $H(x)=\frac{1}{n}\sum_{i=1}^{n} f_i(x)$  and $F(Ax)=\lambda_1 \| Ax\|_1 $ . Subsequently, we scrutinize four publicly available datasets \cite{Chang2011LIBSVM} as illustrated in Table \ref{tab:graph_fused_lasso_datasets}.

\begin{table}
\caption{Real datasets for graph - guided fused lasso.}
\centering
\begin{tabular}{|c|c|c|c|c|c|}
\hline
datasets & \ training & \ test & \ features & \ classes\\
\hline
splice-scale & 500 & 500 & 60 & 2  \\
\hline
a8a & 11348 & 11348 & 300 & 2 \\
\hline
a9a & 16280 & 16280 & 123 & 2 \\
\hline
ijcnn1 & 24995 & 24995 & 22 & 2 \\
\hline

\end{tabular}
\label{tab:graph_fused_lasso_datasets}
\end{table}

{In the experimental setup, to validate the SFO complexity of the proposed algorithm, we compare our algorithm SMADMM with three other stochastic ADMM algorithms, including  SADMM \cite{huang2016stochastic}, SVRG-ADMM \cite{huang2016stochastic}, SARAH-ADMM \cite{huang2019faster} and ASVRG-ADMM \cite{zeng2024accelerated}.  All algorithms are implemented in MATLAB, and all experiments are performed on a PC with an Intel i7-4790 CPU and 16GB memory. }

{All experiments used fixed regularization $\lambda_1 = 10^{-11}$ with batch sizes varying by algorithm: SADMM/SMADMM employed adaptive batches $\{100, 200, 300\}$ based on dataset dimensions, while SVRG-ADMM/SARAH-ADMM/ASVRG-ADMM utilized full outer gradients with fixed inner-loop batches of 300 \cite{Bian_2021}. Parameter optimization used grid search over theoretically valid ranges for step size coefficients ($c_\eta \in [0.05, 0.3]$) and momentum weights ($a_k \in [0.01, 1.0]$). }

For SMADMM specifically, the adaptive step size followed $\eta_k = \min(0.1k^{1/3},\ 0.5)$ with $c_\eta = 0.1$, while momentum decay implemented $a_k = \max(0.5k^{-2/3},\ 0.01)$. 

\begin{figure}[htbp]
    \centering
    \captionsetup[subfigure]{justification=centering}
    
    \begin{subfigure}[t]{0.46\linewidth}
        \centering
        \includegraphics[height=6cm]{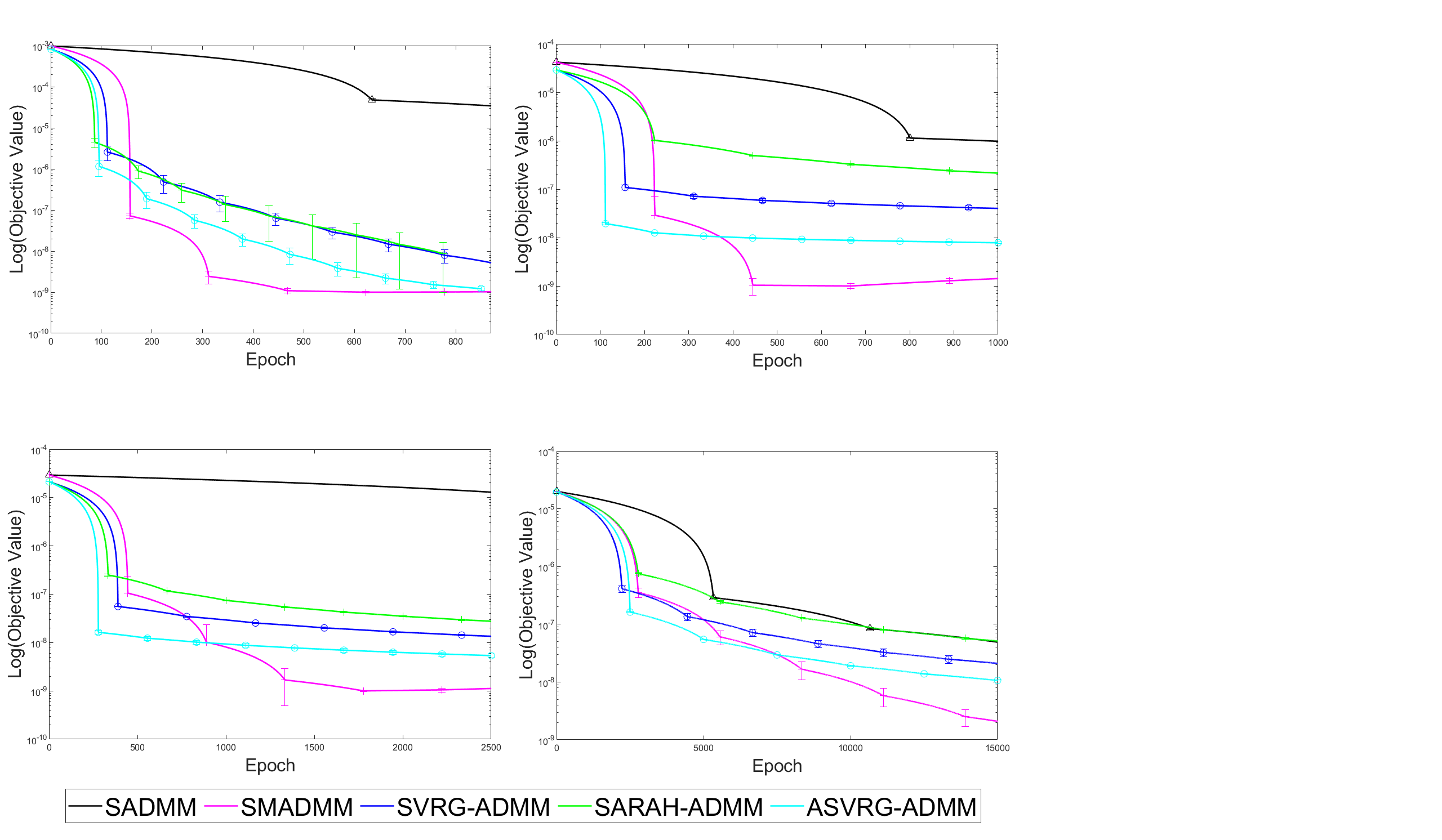}
        \caption{Objective value.}
        \label{fig:objective-value}
    \end{subfigure}
    \hfill
    \begin{subfigure}[t]{0.46\linewidth}
        \centering
        \includegraphics[height=6cm]{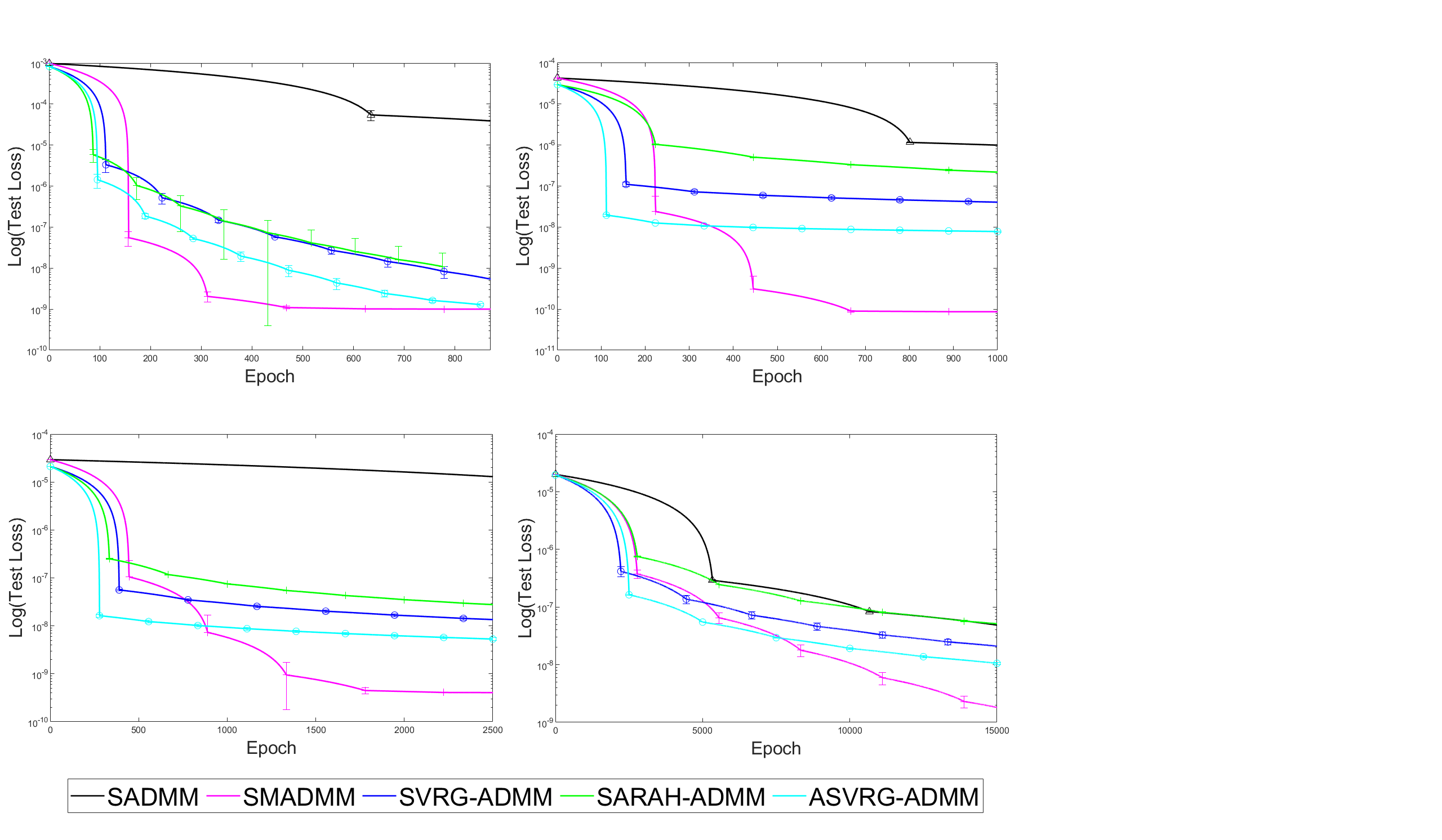}
        \caption{Test loss.}
        \label{fig:test-loss}
    \end{subfigure}

    \caption{Comparison of epoch-wise trends for five algorithms across four datasets.}
    \label{fig:combined-epoch-plots}
\end{figure}

{To comprehensively analyze the SFO complexity-performance relationship, we initially conducted dual evaluation of objective function values and test loss against both CPU time and training epochs. Observing strong correlation between epoch-based and time-based progression trends, we present only the epoch-normalized results in Figure \ref{fig:combined-epoch-plots} to avoid redundancy. These figures demonstrate SMADMM's superior convergence speed and accuracy across all datasets (splice-scale, a8a, a9a, ijcnn1).}

\subsection{Sparse-View CT reconstruction} 
Now we consider a sparse-view Computed Tomography
(SVCT) measurement model \cite{kak2001principles}:
\begin{equation}
    \min_{x\in \mathbb{R}^{n}} \frac{1}{N} \sum_{i=1}^N \|A_i x - b_i\|^2 + r(x),
\end{equation}
where $b_i\in \mathbb{R}^{m}$ is the measured sinogram for the $i$-th projection, $A_i$ is a discretized Radon transform matrix of size $m\times n$ corresponding to a parallel beam setting, $x\in\mathbb{R}^n$ is the image, $R(x)$ denote the regularization function.  We consider simulated data obtained from the clinically realistic CT images provided by Mayo
Clinic for the low-dose CT grand challenge~\cite{Mayo2016}. We compare our PnP-SMADMM algorithm with other ADMM algorithms with PnP prior.   
Specifically, 5936 2D slices of size $512\times 512$ are used to train the models. Another 10 slides are used for testing. The training CT images are divided into $128\times 128$ patches. We use DnCNN~\cite{Zhang2017} as the denoiser $D_\theta$ with the fixed noise level $\sigma = 5$, which consists of 17 convolutional layers. In order to ensure differentiability, we change RELU activations $\mathrm{RELU}(x) =\max \{x,  0\}$ to $\mathrm{Softplus}(x) = \ln (1+\exp (x))$. We aim to train a gradient step denoiser $D_\theta:\mathbb{R}^n\to \mathbb{R}^n$, i.e., $	D_\theta(\mathbf{x}) = \mathbf{x} -  \nabla g_\theta(\mathbf{x})$, where $\nabla g_\theta:\mathbb{R}^n\to \mathbb{R}$ is a scalar function parameterized by a differentiable neural network.  
 The gradient denoiser was trained using the Adam optimizer for 50 epochs, the batch size was set to 128. The learning rate was set to $10^{-3}$ for the first 25 epoches and then reduced to $10^{-4}$. The denoiser $D_\theta$ was trained on a single NVIDIA A800 80GB GPU, and it took about 6.4 hours. All algorithms were implemented under the open-source deep learning framework PyTorch. 

\begin{table*}[!ht]
    \centering\footnotesize\setlength\tabcolsep{6.pt}
    \caption{Average SNR and SSIM values compared with different methods on the 5 test  slides with input $\text{SNR}=50$ dB and total $120$ projection views.}
    \label{table: 120 views}
    \begin{tabular}{lcccccccc}
        \toprule
        \multirow{2}{*}{Methods}
        & \multicolumn{2}{c}{5 batch sizes} & \multicolumn{2}{c}{10 batch sizes} &	\multicolumn{2}{c}{20 batch sizes} & \multicolumn{2}{c}{40 batch sizes}\\
        % \cmidrule(lr){2-5} \cmidrule(lr){6-9}
		\cmidrule(lr){2-3} \cmidrule(lr){4-5} \cmidrule(lr){6-7} \cmidrule(lr){8-9}
        & SNR & SSIM & SNR & SSIM  & SNR & SSIM & SNR & SSIM  \\
        \midrule
        RED-SD~\cite{romano2017little}  & 32.27 & 0.9679 & 32.34 & 0.9688 & 32.36 & 0.9691 & 32.36 & 0.9691  \\
        SPIDER-ADMM~\cite{huang2019faster} & 32.80 & 0.9676 &
        32.91 & 0.9686 & 33.07 & 0.9701 & 33.12 & 0.9707 \\
        PnP-SADMM~\cite{sun2021scalable}  & 33.05 & 0.9697 & 33.14 & 0.9707 & 33.18 & 0.9711 & 33.20 & 0.9713  \\
        % SGDNet  & 30.51 & 0.9494 & 30.54 & 0.9498 & 30.56 & 0.9501 & 30.56 & 0.9503  \\
        % PnP-SMADMM & {\bf 33.17} & {\bf 0.9712} & {\bf 33.20} & {\bf 0.9713}  & {\bf 33.20} & {\bf 0.9714} & {\bf 33.21} &  {\bf 0.9714}\\
        ASVRG-ADMM~\cite{zeng2024accelerated} & 32.96 & 0.9697 & 33.05 & 0.9706 & 33.07 & 0.9709 & 33.09 & 0.9710 \\
        PnP-SMADMM & {\bf 33.17} & {\bf 0.9710} & {\bf 33.19} & {\bf 0.9713}  & {\bf 33.20} & {\bf 0.9714} & {\bf 33.21} &  {\bf 0.9714}\\
        \bottomrule
    \end{tabular}
\end{table*}
\begin{table*}[!ht]
    \centering\footnotesize\setlength\tabcolsep{6.pt}
    \caption{Average SNR and SSIM values compared with different methods on the 5 test  slides with input $\text{SNR}=50$ dB and total $180$ projection views.}
    \label{table: 180 views}
    \begin{tabular}{lcccccccc}
        \toprule
        \multirow{2}{*}{Methods}
        & \multicolumn{2}{c}{5 batch sizes} & \multicolumn{2}{c}{10 batch sizes} &	\multicolumn{2}{c}{20 batch sizes} & \multicolumn{2}{c}{40 batch sizes}\\
        % \cmidrule(lr){2-5} \cmidrule(lr){6-9}
		\cmidrule(lr){2-3} \cmidrule(lr){4-5} \cmidrule(lr){6-7} \cmidrule(lr){8-9}
        & SNR & SSIM & SNR & SSIM  & SNR & SSIM & SNR & SSIM  \\
        \midrule
        RED-SD~\cite{romano2017little}  & 33.05 & 0.9726 & 33.16 & 0.9737 & 33.21 & 0.9742 & 33.24 & 0.9745  \\
        SPIDER-ADMM~\cite{huang2019faster} & 33.71 & 0.9722 & 33.95 & 0.9738 & 34.14 & 0.9754 & 34.32 & 0.9765 \\
        PnP-SADMM~\cite{sun2021scalable}  & 34.17 & 0.9753 & 34.33 & 0.9765 & {34.40} & {0.9771} & 34.45 & 0.9774  \\
        % SGDNet  & 34.67 & 0.9817 & 34.73 & 0.9820 & {34.78} & {0.9822} & 34.79 & 0.9823  \\
        % PnP-SMADMM & {\bf 34.42} & {\bf 0.9773} & {\bf 34.45} & {\bf 0.9775}  & {\bf 34.47} & {\bf 0.9776} & {\bf 34.48} &  {\bf 0.9776}\\
        ASVRG-ADMM~\cite{zeng2024accelerated} & 34.21 & 0.9758 & 34.33 & 0.9767 & 34.40 & 0.9772 & 34.43 & 0.9774 \\
        PnP-SMADMM & {\bf 34.38} & {\bf 0.9770} & {\bf 34.43} & {\bf 0.9773}  & {\bf 34.46} & {\bf 0.9775} & {\bf 34.48} &  {\bf 0.9776}\\
        \bottomrule
    \end{tabular}
\end{table*}

\begin{table*}[!ht]
    \centering\footnotesize\setlength\tabcolsep{10.pt}
    \caption{Average SNR and SSIM values about different $a_k = \frac{1}{k^\alpha} (\alpha =0.1,0.5,2/3,2)$ on the 5 test slides with input $\text{SNR}=50$ dB and total $180$ projection views.}
    \label{table: ablation study on ak}
    \begin{tabular}{lcccccccc}
        \toprule
        \multirow{2}{*}{Parameters}
        & \multicolumn{2}{c}{5 batch sizes} & \multicolumn{2}{c}{10 batch sizes} &	\multicolumn{2}{c}{20 batch sizes} & \multicolumn{2}{c}{40 batch sizes}\\
        % \cmidrule(lr){2-5} \cmidrule(lr){6-9}
		\cmidrule(lr){2-3} \cmidrule(lr){4-5} \cmidrule(lr){6-7} \cmidrule(lr){8-9}
        & SNR & SSIM & SNR & SSIM  & SNR & SSIM & SNR & SSIM  \\
        \midrule
        $\alpha = 0.1$  & 34.19 & 0.9733 & 34.34 & 0.9767 & {34.41} & {0.9771} & 34.45 & 0.9774  \\
		$\alpha = 0.5$  & 34.32 & 0.9744 & 34.40 & 0.9771 & {34.44} & {0.9774} & 34.47 & 0.9775  \\
        $\alpha = 2/3 $ & {\bf 34.38} & {\bf 0.9770} & {\bf 34.43} & {\bf 0.9773}  & {\bf 34.46} & {\bf 0.9775} & {\bf 34.48} &  {\bf 0.9776}\\
        % $\alpha = 1$ & {\bf 34.42} & {\bf 0.9773} & {\bf 34.45} & {\bf 0.9775}  & {\bf 34.47} & {\bf 0.9776} & {\bf 34.48} &  {\bf 0.9776}\\
		$\alpha = 2$  & 23.55 & 0.9162 & 27.29 & 0.9259 & {30.61} & {0.9552} & 32.44 & 0.9675 \\
        \bottomrule
    \end{tabular}
\end{table*}
 
\begin{table*}[!ht]
    \centering\footnotesize\setlength\tabcolsep{6.pt}
    \caption{Average SNR value compared with different methods on the 3 test  images with input $\text{SNR}=25$ dB and total $6$ measurements for phase retrieval.}
    \label{table: 6 measurements}
    \begin{tabular}{lccccc}
        \toprule
        Method
        & $B=1$ & $B=2$ & $B=3$ \\
        \midrule
        On-RED~\cite{wu2019online}  & 31.33 & 32.06 & 32.07 \\
        SPIDER-ADMM~\cite{huang2019faster} & 31.37 & 33.21 & 33.03 \\
        PnP-SADMM~\cite{sun2021scalable}  & 31.26 & 33.03 & 33.06  \\
        ASVRG-ADMM~\cite{zeng2024accelerated} & 31.52 & 33.27 & 33.20 \\
        PnP-SMADMM & {\bf 31.56} & {\bf 33.76} & {\bf 33.57} \\
        \bottomrule
    \end{tabular}
\end{table*}

We implement the measurement operator $A_i$ and its adjoint $A_i^{\mathrm{T}}$ using the PyTorch implementations of the \texttt{Radon} and \texttt{IRadon}~\footnote{https://github.com/phernst/pytorch\_radon} transforms. The CT machine is assumed to project from nominal angles with $N \in \{ 120, 180 \}$ projection views, which are evenly distributed over a half circle, using 724 detector pixels. Gaussian noise is added to the sinograms to achieve an input SNR of 50 dB. We compare the classic PnP-SADMM method, which is a special case of PnP-SMADMM with $a=1$, and SPIDER-ADMM with the same PnP prior, other parameters, including the step size $\eta$ and the penalty coefficient $\rho$ are the same. {For parameter selection, according to Theorem 3.2, we choose the optimal $a_k = 1/k^{\frac{2}{3}}$.
Table~\ref{table: 120 views} and Table~\ref{table: 180 views} show the average SNR and SSIM values of RED-SD(steepest descent)~\cite{romano2017little}, SPIDER-ADMM~\cite{huang2019faster}, the classic PnP-SADMM method~\cite{sun2021scalable}, ASVRG-ADMM~\cite{zeng2024accelerated}, and the proposed method on the 10 test slides with input $\text{SNR}=50$ dB and total $120$ and $180$ projection views, respectively. The batch size is set to 5, 10, 20, and 40. The results show that the proposed method outperforms the classic PnP-SADMM method in terms of both SNR and SSIM. The proposed method achieves better and more stable recovery results than the classic method with minibatch sizes, and it has the memory efficient advantage due to its fewer online measurements. The visual comparison of the 180 views CT reconstruction with RED-SD  and PnP-SADMM is shown in Figure~\ref{fig: 180 views with batch size 5}. The results show that the proposed method can achieve better image quality than the classic PnP-SADMM method and RED-SD. Recovery results over iteration about the classic PnP-SADMM method and the proposed method with 5 minibatch sizes are shown in Figure~\ref{fig: time complexity}. These results show that the proposed method with the minibatch size achieves superior performance against the classic PnP-SADMM method. The ablation study on $a_k = \frac{1}{k^\alpha} (\alpha =0.1,0.5,2/3,2)$ is shown in Table~\ref{table: ablation study on ak}, the case $ a_k = \frac{1}{k^{2/3}}$ chieves the best performance. The numerical results are consistent with Theorem~\ref{theorem:diminish}.
\begin{figure}[htbp]
	\centering
	\captionsetup[subfigure]{justification=centering}
	\begin{subfigure}[b]{.49\linewidth}
		\centering
		\begin{tikzpicture}[spy using outlines={rectangle,blue,magnification=5,size=1.2cm, connect spies}]
		\node {\includegraphics[height=3cm]{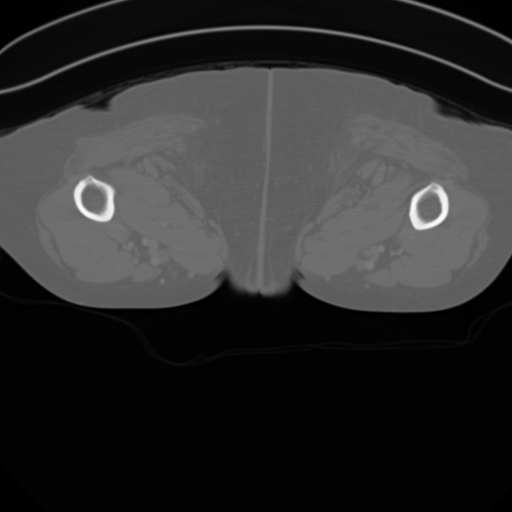}};
		\spy on (-0.8,0.2) in node [left] at (1.5,.9);
		\end{tikzpicture}
		\caption{Ground-truth}
	\end{subfigure}
	\begin{subfigure}[b]{.49\linewidth}
		\centering
		\begin{tikzpicture}[spy using outlines={rectangle,blue,magnification=5,size=1.2cm, connect spies}]
		\node {\includegraphics[height=3cm]{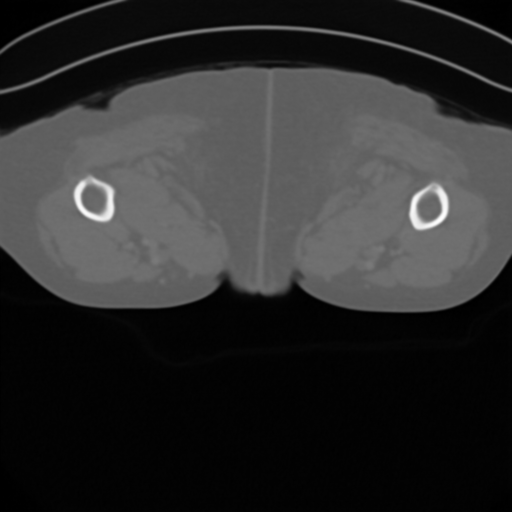}};
		\spy on (-0.8,0.2) in node [left] at (1.5,.9);
		\end{tikzpicture}
		\caption{RED-SD (34.15 dB) }
	\end{subfigure}
	\\
	\begin{subfigure}[b]{.49\linewidth}
		\centering
		\begin{tikzpicture}[spy using outlines={rectangle,blue,magnification=5,size=1.2cm, connect spies}]
		\node {\includegraphics[height=3cm]{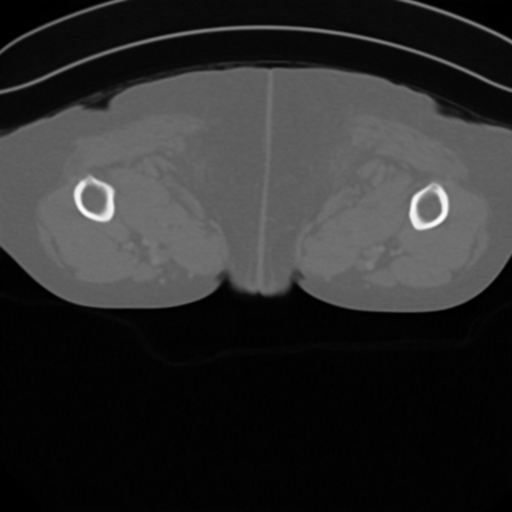}};
		\spy on (-0.8,0.2) in node [left] at (1.5,.9);
		\end{tikzpicture}
		\caption{PnP-SADMM \\ (35.10 dB)}
	\end{subfigure}
	\begin{subfigure}[b]{.49\linewidth}
		\centering
		\begin{tikzpicture}[spy using outlines={rectangle,blue,magnification=5,size=1.2cm, connect spies}]
		\node {\includegraphics[height=3cm]{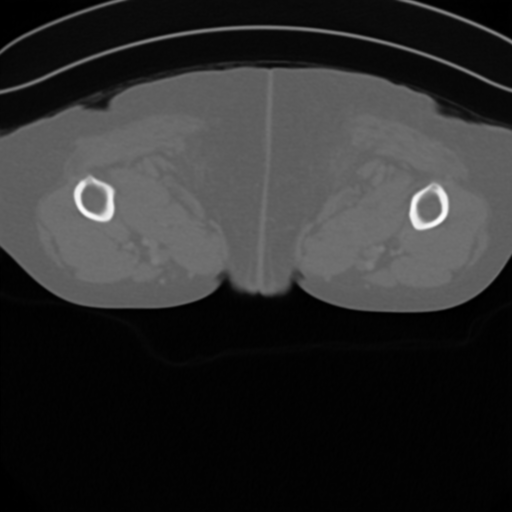}};
		\spy on (-0.8,0.2) in node [left] at (1.5,.9);
		\end{tikzpicture}
		\caption{PnP-SMADMM \\ (35.36 dB) }
	\end{subfigure}
	\caption{Visual comparison of 180 views CT reconstruction with RED-SD and PnP-SADMM. The input SNR is $50$ dB, and the batch size is set to 5.}
	\label{fig: 180 views with batch size 5}
\end{figure}
\begin{figure}[htbp]
	\centering
	\captionsetup[subfigure]{justification=centering}
	\begin{subfigure}[b]{0.49\linewidth}
		\centering
		\includegraphics[height=3cm]{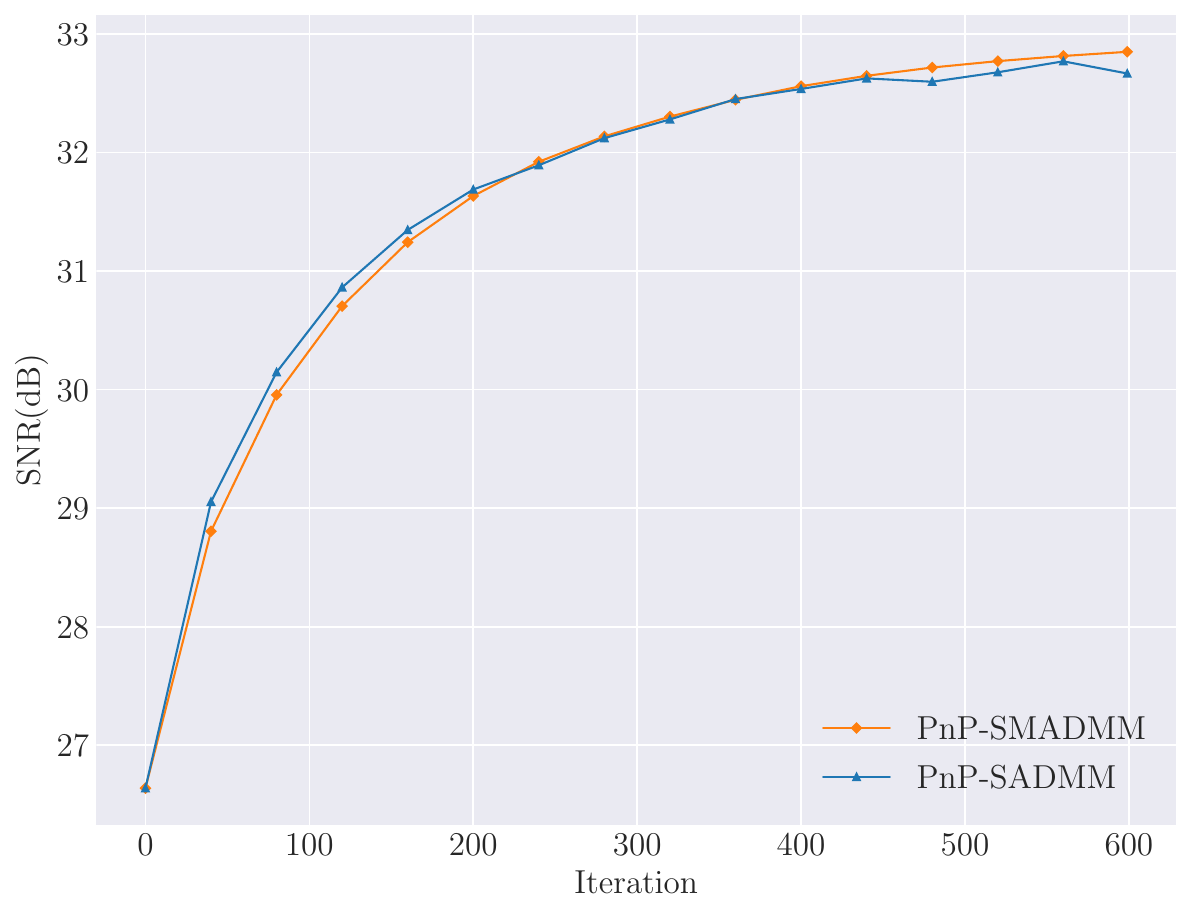}
		\caption{120 views}
		\label{subfig:120 views}
	\end{subfigure}
	\begin{subfigure}[b]{0.49\linewidth}
		\centering
		\includegraphics[height=3cm]{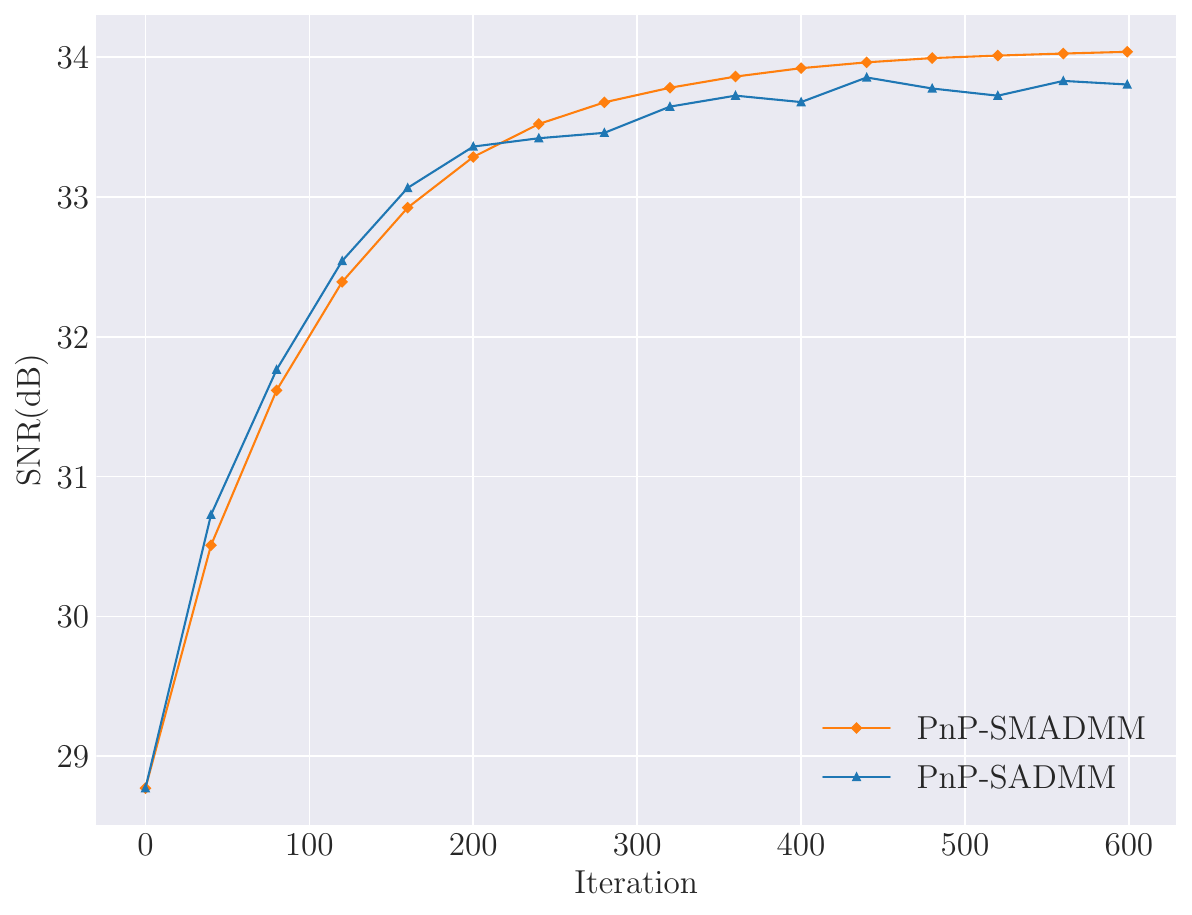}
		\caption{180 views}
		\label{subfig:180 views}
	\end{subfigure}
 \begin{subfigure}[b]{0.49\linewidth}
		\centering
		\includegraphics[height=3cm]{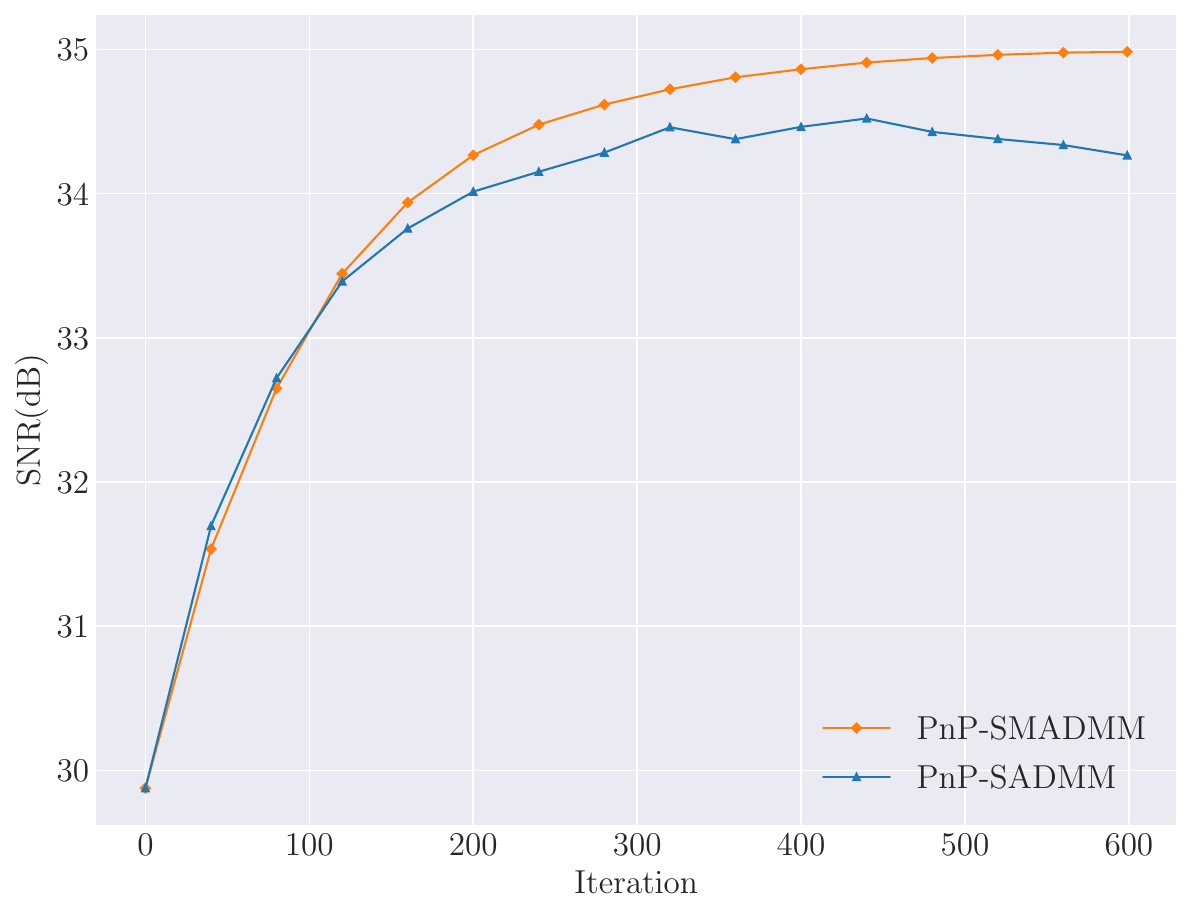}
		\caption{360 views}
		\label{subfig:360 views}
	\end{subfigure}
 \begin{subfigure}[b]{0.49\linewidth}
		\centering
		\includegraphics[height=3cm]{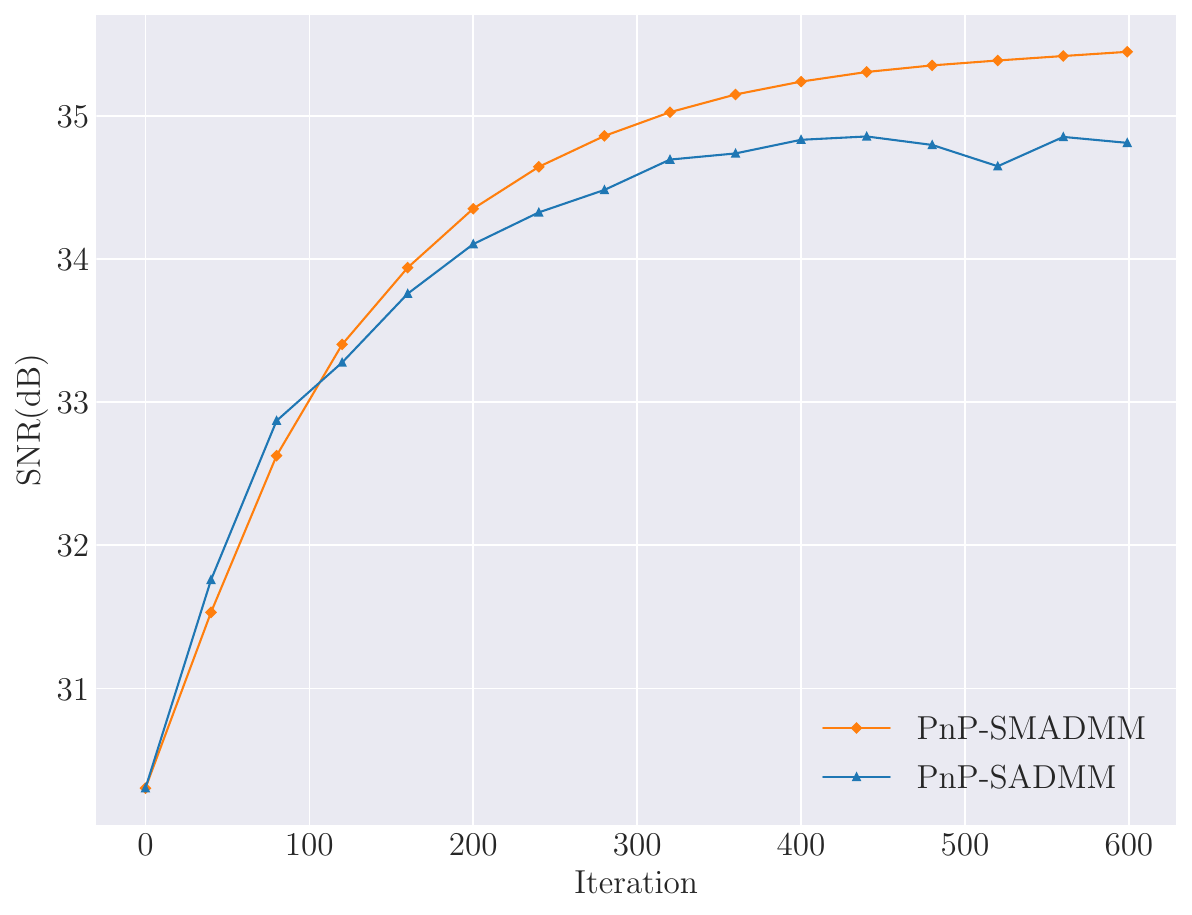}
		\caption{720 views}
		\label{subfig:720 views}
	\end{subfigure}
	\caption{Performance Comparison of CT image reconstruction over iterations with 5 minibatch sizes.}
	\label{fig: time complexity}
\end{figure}

\subsection{Phase Retrieval}
We evaluated the performance of PnP-SMADMM on a nonconvex phase retrieval problem~\eqref{problem:PR} using coded diffraction patterns (CDP), formulated as:
\begin{equation}\label{problem:PR}
    \min_{x\in \mathbb{R}^n} \frac{1}{N} \sum_{i=1}^N \left\| \mathbf{y}_i -\left\vert \mathcal{F}\mathcal{M}_i \mathbf{x} \right\vert \right\|^2 + r(\mathbf{x}),
\end{equation}
where $\mathcal{F}$ represents the 2D discrete Fast Fourier Transform (FFT), and $\mathcal{M}_i$ is the $i$-th random phase mask that modulates the light and the modulation code. Each entry of $\mathcal{M}_i$ is uniformly drawn from the unit circle in the complex plane. We compare On-RED~\cite{wu2019online}, the classic PnP-SADMM~\cite{sun2021scalable}, SPIDER-ADMM~\cite{huang2019faster}, and ASVRG-ADMM~\cite{zeng2024accelerated} with PnP priors. To ensure a fair comparison, all hyperparameters were kept consistent across online ADMM algorithms, with $\eta = \frac{1}{6+2\tau},\tau = 1.3181, K =600$, after careful manual tuning. Table~\ref{table: 6 measurements} presents the comparison with state-of-the-art online methods incorporating PnP priors, PnP-SMADMM achieves the best performance.

\section{Conslusion}
This paper introduces a single-loop  SMADMM for tackling a class of nonconvex and nonsmooth optimization problems. We establish that SMADMM achieves an optimal oracle complexity of $\mathcal{O}(\epsilon^{-\frac{3}{2}})$ in the online setting. In particular, SMADMM requires only $\mathcal{O}(1)$ gradient evaluations per iteration and avoids the need for restarting with large batch gradients.  Furthermore, we extend our method by integrating it with PnP priors, resulting in the PnP-SMADMM algorithm. Numerical experiments on classification, CT image reconstruction and phase retrieve validate the theoretical findings. Finally, our proposed algorithms can easily extended to solve the following multi-block optimization problem.

\bibliographystyle{siamplain}
\bibliography{ref}

\begin{thebibliography}{10}

\bibitem{ahmad2020plug}
{\sc R.~Ahmad, C.~A. Bouman, G.~T. Buzzard, S.~Chan, S.~Liu, E.~T. Reehorst,
  and P.~Schniter}, {\em Plug-and-play methods for magnetic resonance imaging:
  Using denoisers for image recovery}, IEEE signal processing magazine, 37
  (2020), pp.~105--116.

\bibitem{bai2022inexact}
{\sc J.~Bai, W.~W. Hager, and H.~Zhang}, {\em An inexact accelerated stochastic
  admm for separable convex optimization}, Computational Optimization and
  Applications, 81 (2022), pp.~479--518.

\bibitem{bai2021convergence}
{\sc J.~Bai, D.~Han, H.~Sun, and H.~Zhang}, {\em Convergence on a symmetric
  accelerated stochastic admm with larger stepsizes}, arXiv preprint
  arXiv:2103.16154,  (2021).

\bibitem{barber2024convergence}
{\sc R.~F. Barber and E.~Y. Sidky}, {\em Convergence for nonconvex admm, with
  applications to ct imaging}, Journal of Machine Learning Research, 25 (2024),
  pp.~1--46.

\bibitem{Bian_2021}
{\sc F.~Bian, J.~Liang, and X.~Zhang}, {\em A stochastic alternating direction
  method of multipliers for non-smooth and non-convex optimization}, Inverse
  Problems, 37 (2021), p.~075009,
  \url{https://doi.org/10.1088/1361-6420/ac0966},
  \url{https://dx.doi.org/10.1088/1361-6420/ac0966}.

\bibitem{boyd2011distributed}
{\sc S.~Boyd, N.~Parikh, E.~Chu, B.~Peleato, J.~Eckstein, et~al.}, {\em
  Distributed optimization and statistical learning via the alternating
  direction method of multipliers}, Foundations and Trends{\textregistered} in
  Machine learning, 3 (2011), pp.~1--122.

\bibitem{Chang2011LIBSVM}
{\sc C.-C. Chang and C.-J. Lin}, {\em Libsvm : a library for support vector
  machines}, ACM Transactions on Intelligent Systems and Technology, 2 (2011),
  pp.~27:1--27:27.

\bibitem{cohen2021has}
{\sc R.~Cohen, Y.~Blau, D.~Freedman, and E.~Rivlin}, {\em It has potential:
  Gradient-driven denoisers for convergent solutions to inverse problems},
  Advances in Neural Information Processing Systems, 34 (2021),
  pp.~18152--18164.

\bibitem{cutkosky2019momentum}
{\sc A.~Cutkosky and F.~Orabona}, {\em Momentum-based variance reduction in
  non-convex sgd}, Advances in neural information processing systems, 32
  (2019).

\bibitem{defazio2014saga}
{\sc A.~Defazio, F.~Bach, and S.~Lacoste-Julien}, {\em Saga: A fast incremental
  gradient method with support for non-strongly convex composite objectives},
  Advances in neural information processing systems, 27 (2014).

\bibitem{fang2017faster}
{\sc C.~Fang, F.~Cheng, and Z.~Lin}, {\em Faster and non-ergodic o (1/k)
  stochastic alternating direction method of multipliers}, Advances in Neural
  Information Processing Systems, 30 (2017).

\bibitem{Friedman2008Sparse}
{\sc J.~Friedman, T.~Hastie, and R.~Tibshirani}, {\em Sparse inverse covariance
  estimation with the graphical lasso}, Biostatistics, 9 (2008), pp.~432--441.

\bibitem{gabay1976dual}
{\sc D.~Gabay and B.~Mercier}, {\em A dual algorithm for the solution of
  nonlinear variational problems via finite element approximation}, Computers
  \& mathematics with applications, 2 (1976), pp.~17--40.

\bibitem{glowinski1975approximation}
{\sc R.~Glowinski and A.~Marroco}, {\em Sur l'approximation, par
  {\'e}l{\'e}ments finis d'ordre un, et la r{\'e}solution, par
  p{\'e}nalisation-dualit{\'e} d'une classe de probl{\`e}mes de dirichlet non
  lin{\'e}aires}, Revue fran{\c{c}}aise d'automatique, informatique, recherche
  op{\'e}rationnelle. Analyse num{\'e}rique, 9 (1975), pp.~41--76.

\bibitem{han2022survey}
{\sc D.-R. Han}, {\em A survey on some recent developments of alternating
  direction method of multipliers}, Journal of the Operations Research Society
  of China,  (2022), pp.~1--52.

\bibitem{he2018optimizing}
{\sc J.~He, Y.~Yang, Y.~Wang, D.~Zeng, Z.~Bian, H.~Zhang, J.~Sun, Z.~Xu, and
  J.~Ma}, {\em Optimizing a parameterized plug-and-play admm for iterative
  low-dose ct reconstruction}, IEEE transactions on medical imaging, 38 (2018),
  pp.~371--382.

\bibitem{huang2018mini}
{\sc F.~Huang and S.~Chen}, {\em Mini-batch stochastic admms for nonconvex
  nonsmooth optimization}, arXiv preprint arXiv:1802.03284,  (2018).

\bibitem{huang2019faster}
{\sc F.~Huang, S.~Chen, and H.~Huang}, {\em Faster stochastic alternating
  direction method of multipliers for nonconvex optimization}, in International
  conference on machine learning, PMLR, 2019, pp.~2839--2848.

\bibitem{huang2016stochastic}
{\sc F.~Huang, S.~Chen, and Z.~Lu}, {\em Stochastic alternating direction
  method of multipliers with variance reduction for nonconvex optimization},
  arXiv preprint arXiv:1610.02758,  (2016).

\bibitem{hurault2021gradient}
{\sc S.~Hurault, A.~Leclaire, and N.~Papadakis}, {\em Gradient step denoiser
  for convergent plug-and-play}, arXiv preprint arXiv:2110.03220,  (2021).

\bibitem{hurault2022proximal}
{\sc S.~Hurault, A.~Leclaire, and N.~Papadakis}, {\em Proximal denoiser for
  convergent plug-and-play optimization with nonconvex regularization}, in
  International Conference on Machine Learning, PMLR, 2022, pp.~9483--9505.

\bibitem{johnson2013accelerating}
{\sc R.~Johnson and T.~Zhang}, {\em Accelerating stochastic gradient descent
  using predictive variance reduction}, Advances in neural information
  processing systems, 26 (2013).

\bibitem{kak2001principles}
{\sc A.~C. Kak and M.~Slaney}, {\em Principles of computerized tomographic
  imaging}, SIAM, 2001.

\bibitem{kamilov2023plug}
{\sc U.~S. Kamilov, C.~A. Bouman, G.~T. Buzzard, and B.~Wohlberg}, {\em
  Plug-and-play methods for integrating physical and learned models in
  computational imaging: Theory, algorithms, and applications}, IEEE Signal
  Processing Magazine, 40 (2023), pp.~85--97.

\bibitem{Kim2009Multivariate}
{\sc S.~Kim, K.-A. Sohn, and E.~P. Xing}, {\em A multivariate regression
  approach to association analysis of a quantitative trait network},
  Bioinformatics, 25 (2009), pp.~i204--i212.

\bibitem{levy2021storm}
{\sc K.~Levy, A.~Kavis, and V.~Cevher}, {\em Storm+: Fully adaptive sgd with
  recursive momentum for nonconvex optimization}, Advances in Neural
  Information Processing Systems, 34 (2021), pp.~20571--20582.

\bibitem{liu2021recovery}
{\sc J.~Liu, S.~Asif, B.~Wohlberg, and U.~Kamilov}, {\em Recovery analysis for
  plug-and-play priors using the restricted eigenvalue condition}, Advances in
  Neural Information Processing Systems, 34 (2021), pp.~5921--5933.

\bibitem{liu2020accelerated}
{\sc Y.~Liu, F.~Shang, H.~Liu, L.~Kong, L.~Jiao, and Z.~Lin}, {\em Accelerated
  variance reduction stochastic admm for large-scale machine learning}, IEEE
  Transactions on Pattern Analysis and Machine Intelligence, 43 (2020),
  pp.~4242--4255.

\bibitem{mancino2023decentralized}
{\sc G.~Mancino-Ball, Y.~Xu, and J.~Chen}, {\em A decentralized primal-dual
  framework for non-convex smooth consensus optimization}, IEEE Transactions on
  Signal Processing, 71 (2023), pp.~525--538.

\bibitem{Mayo2016}
{\sc C.~McCollough}, {\em Tu-fg-207a-04: Overview of the low dose ct grand
  challenge}, Medical Physics, 43 (2016), pp.~3759--3760,
  \url{https://doi.org/https://doi.org/10.1118/1.4957556}.

\bibitem{metzler2018prdeep}
{\sc C.~Metzler, P.~Schniter, A.~Veeraraghavan, and R.~Baraniuk}, {\em prdeep:
  Robust phase retrieval with a flexible deep network}, in International
  Conference on Machine Learning, PMLR, 2018, pp.~3501--3510.

\bibitem{mirzaeifard2024decentralized}
{\sc R.~Mirzaeifard, D.~Ghaderyan, and S.~Werner}, {\em Decentralized smoothing
  admm for quantile regression with non-convex sparse penalties}, arXiv
  preprint arXiv:2408.01307,  (2024).

\bibitem{Nguyen2017}
{\sc L.~M. Nguyen, J.~Liu, K.~Scheinberg, and M.~Tak\'{a}\v{c}}, {\em Sarah: a
  novel method for machine learning problems using stochastic recursive
  gradient}, in Proceedings of the 34th International Conference on Machine
  Learning - Volume 70, ICML'17, JMLR.org, 2017, p.~2613–2621.

\bibitem{ouyang2013stochastic}
{\sc H.~Ouyang, N.~He, L.~Tran, and A.~Gray}, {\em Stochastic alternating
  direction method of multipliers}, in International conference on machine
  learning, PMLR, 2013, pp.~80--88.

\bibitem{reehorst2018regularization}
{\sc E.~T. Reehorst and P.~Schniter}, {\em Regularization by denoising:
  Clarifications and new interpretations}, IEEE transactions on computational
  imaging, 5 (2018), pp.~52--67.

\bibitem{romano2017little}
{\sc Y.~Romano, M.~Elad, and P.~Milanfar}, {\em The little engine that could:
  Regularization by denoising (red)}, SIAM Journal on Imaging Sciences, 10
  (2017), pp.~1804--1844.

\bibitem{ryu2019plug}
{\sc E.~Ryu, J.~Liu, S.~Wang, X.~Chen, Z.~Wang, and W.~Yin}, {\em Plug-and-play
  methods provably converge with properly trained denoisers}, in International
  Conference on Machine Learning, PMLR, 2019, pp.~5546--5557.

\bibitem{sun2019online}
{\sc Y.~Sun, B.~Wohlberg, and U.~S. Kamilov}, {\em An online plug-and-play
  algorithm for regularized image reconstruction}, IEEE Transactions on
  Computational Imaging, 5 (2019), pp.~395--408.

\bibitem{sun2021scalable}
{\sc Y.~Sun, Z.~Wu, X.~Xu, B.~Wohlberg, and U.~S. Kamilov}, {\em Scalable
  plug-and-play admm with convergence guarantees}, IEEE Transactions on
  Computational Imaging, 7 (2021), pp.~849--863.

\bibitem{suzuki2013dual}
{\sc T.~Suzuki}, {\em Dual averaging and proximal gradient descent for online
  alternating direction multiplier method}, in International Conference on
  Machine Learning, PMLR, 2013, pp.~392--400.

\bibitem{suzuki2014stochastic}
{\sc T.~Suzuki}, {\em Stochastic dual coordinate ascent with alternating
  direction method of multipliers}, in International Conference on Machine
  Learning, PMLR, 2014, pp.~736--744.

\bibitem{tang2020fast}
{\sc J.~Tang and M.~Davies}, {\em A fast stochastic plug-and-play admm for
  imaging inverse problems}, arXiv preprint arXiv:2006.11630,  (2020).

\bibitem{terris2020building}
{\sc M.~Terris, A.~Repetti, J.-C. Pesquet, and Y.~Wiaux}, {\em Building firmly
  nonexpansive convolutional neural networks}, in ICASSP 2020-2020 IEEE
  International Conference on Acoustics, Speech and Signal Processing (ICASSP),
  IEEE, 2020, pp.~8658--8662.

\bibitem{venkatakrishnan2013plug}
{\sc S.~V. Venkatakrishnan, C.~A. Bouman, and B.~Wohlberg}, {\em Plug-and-play
  priors for model based reconstruction}, in 2013 IEEE global conference on
  signal and information processing, IEEE, 2013, pp.~945--948.

\bibitem{wang2013online}
{\sc H.~Wang and A.~Banerjee}, {\em Online alternating direction method (longer
  version)}, arXiv preprint arXiv:1306.3721,  (2013).

\bibitem{wei2020tuning}
{\sc K.~Wei, A.~Aviles-Rivero, J.~Liang, Y.~Fu, C.-B. Sch{\"o}nlieb, and
  H.~Huang}, {\em Tuning-free plug-and-play proximal algorithm for inverse
  imaging problems}, in International Conference on Machine Learning, PMLR,
  2020, pp.~10158--10169.

\bibitem{wu2019online}
{\sc Z.~Wu, Y.~Sun, J.~Liu, and U.~Kamilov}, {\em Online regularization by
  denoising with applications to phase retrieval}, in Proceedings of the
  IEEE/CVF International Conference on Computer Vision Workshops, 2019,
  pp.~0--0.

\bibitem{xu2015augmented}
{\sc J.~Xu, S.~Zhu, Y.~C. Soh, and L.~Xie}, {\em Augmented distributed gradient
  methods for multi-agent optimization under uncoordinated constant stepsizes},
  in IEEE Conference on Decision and Control, 2015, pp.~2055--2060.

\bibitem{xu2017admm}
{\sc Y.~Xu, M.~Liu, Q.~Lin, and T.~Yang}, {\em Admm without a fixed penalty
  parameter: Faster convergence with new adaptive penalization}, Advances in
  neural information processing systems, 30 (2017).

\bibitem{yang2022survey}
{\sc Y.~Yang, X.~Guan, Q.-S. Jia, L.~Yu, B.~Xu, and C.~J. Spanos}, {\em A
  survey of admm variants for distributed optimization: Problems, algorithms
  and features}, arXiv preprint arXiv:2208.03700,  (2022).

\bibitem{zeng2024unified}
{\sc Y.~Zeng, J.~Bai, S.~Wang, and Z.~Wang}, {\em A unified inexact stochastic
  admm for composite nonconvex and nonsmooth optimization}, arXiv preprint
  arXiv:2403.02015,  (2024).

\bibitem{zeng2024accelerated}
{\sc Y.~Zeng, Z.~Wang, J.~Bai, and X.~Shen}, {\em An accelerated stochastic
  admm for nonconvex and nonsmooth finite-sum optimization}, Automatica, 163
  (2024), p.~111554.

\bibitem{zhang2021plug}
{\sc K.~Zhang, Y.~Li, W.~Zuo, L.~Zhang, L.~Van~Gool, and R.~Timofte}, {\em
  Plug-and-play image restoration with deep denoiser prior}, IEEE Transactions
  on Pattern Analysis and Machine Intelligence, 44 (2021), pp.~6360--6376.

\bibitem{Zhang2017}
{\sc K.~Zhang, W.~Zuo, Y.~Chen, D.~Meng, and L.~Zhang}, {\em Beyond a gaussian
  denoiser: Residual learning of deep cnn for image denoising}, IEEE
  Transactions on Image Processing, 26 (2017), pp.~3142--3155,
  \url{https://doi.org/10.1109/TIP.2017.2662206}.

\bibitem{zhang2017learning}
{\sc K.~Zhang, W.~Zuo, S.~Gu, and L.~Zhang}, {\em Learning deep cnn denoiser
  prior for image restoration}, in Proceedings of the IEEE conference on
  computer vision and pattern recognition, 2017, pp.~3929--3938.

\bibitem{zhang2019deep}
{\sc K.~Zhang, W.~Zuo, and L.~Zhang}, {\em Deep plug-and-play super-resolution
  for arbitrary blur kernels}, in Proceedings of the IEEE/CVF conference on
  computer vision and pattern recognition, 2019, pp.~1671--1681.

\bibitem{zheng2016fast}
{\sc S.~Zheng and J.~T. Kwok}, {\em Fast-and-light stochastic admm.}, in IJCAI,
  2016, pp.~2407--2613.

\bibitem{zheng2016stochastic}
{\sc S.~Zheng and J.~T. Kwok}, {\em Stochastic variance-reduced admm}, arXiv
  preprint arXiv:1604.07070,  (2016).

\bibitem{zhong2014fast}
{\sc W.~Zhong and J.~Kwok}, {\em Fast stochastic alternating direction method
  of multipliers}, in International conference on machine learning, PMLR, 2014,
  pp.~46--54.

\end{thebibliography}

\end{document}